\documentclass[a4paper,11pt]{article}
\usepackage{mathrsfs}
\usepackage{amsmath}                    
\usepackage{amssymb}                    
\usepackage{amsthm}                     
\usepackage{amsfonts}                   
\usepackage{color}
\usepackage{graphics,graphicx}
\usepackage{enumerate}
\usepackage{epsfig}
\usepackage{caption}
\usepackage{subcaption}
\usepackage{array}

\usepackage{multirow}
\usepackage[english]{algorithm}
\usepackage{dsfont}

\usepackage{hyperref}

\allowdisplaybreaks

\newtheorem{theorem}{Theorem}[section]

\newtheorem{proposition}[theorem]{Proposition}
\newtheorem{lemma}[theorem]{Lemma}
\newtheorem{remark}[theorem]{Remark}

\newtheorem{alg}[theorem]{Algorithm}

\numberwithin{equation}{section}
\renewcommand{\epsilon}{\varepsilon}

\newcommand{\bigo}{\mathcal{O}}
\newcommand{\C}{\mathbb{C}}

\newcommand{\e}{\mathrm{e}}
\newcommand{\dd}{\mathrm{ds}}
\title{
Superconvergence of the Strang splitting when using the Crank-Nicolson scheme for parabolic PDEs with Dirichlet and oblique boundary conditions
}

\author{ 
Guillaume Bertoli\textsuperscript{1}, Christophe Besse\textsuperscript{2}, and Gilles Vilmart\textsuperscript{1}
}

\begin{document}
\footnotetext[1]{
Universit\'e de Gen\`eve, Section de math\'ematiques, UNI DUFOUR, 24, rue du Général Dufour, Case postale 64, 1211 Gen\`eve 4, Switzerland, Guillaume.Bertoli@unige.ch, gilles.vilmart@unige.ch}
\footnotetext[2]{Institut de Mathématiques de Toulouse, U.M.R CNRS 5219, Universit\'e de Toulouse, CNRS, UPS IMT, 118 Route de Narbonne, 31062 Toulouse Cedex 9, France, christophe.besse@math.univ-toulouse.fr}
\maketitle

\begin{abstract}
We show that the Strang splitting method applied to a diffusion-reaction equation with inhomogeneous general oblique boundary conditions is of order two when the diffusion equation is solved with the Crank-Nicolson method, while order reduction occurs in general if using other Runge-Kutta schemes or even the exact flow itself for the diffusion part. We prove these results when the source term only depends on the space variable, an assumption which makes the splitting scheme equivalent to the Crank-Nicolson method itself applied to the whole problem. Numerical experiments suggest that the second order convergence persists with general nonlinearities.

\medskip


\medskip


\end{abstract}

\section{Introduction}
We consider a parabolic semilinear differential problem, in a smooth bounded domain $\Omega$ in $\mathbb{R}^d$ in dimension $d\geq 1$, for $t\in [0,T]$, of the form 
\begin{align}\label{eq:parabolic}
\partial_t u(x,t) &= Du(x,t) + f(x,u(x,t))\ \ \mathrm{in}\ \Omega\times (0, T],\nonumber\\
 \quad Bu(x,t) &= b(x)\ \ \mathrm{on}\ \partial\Omega\times (0,T], \nonumber\\ 
 u(x,0)&=u_0(x)\quad \mathrm{in}\ \Omega,
\end{align}
where, for $1<p<\infty$, $D:W^{2,p}(\Omega)\rightarrow L^p(\Omega)$ is a linear diffusion operator and $f:L^p(\Omega)\rightarrow L^p(\Omega)$ is a possibly nonlinear source term. The operator $B$ represents boundary conditions of type Dirichlet, Neumann or Robin. When time-discretizing a problem of this form, it can be advantageous to use a splitting method in order to divide the problem~\eqref{eq:parabolic} into two parts:  
the source equation
\begin{align}\label{eq:source}
\partial_t u(x,t) = f(x,u(x,t))\quad \mathrm{in}\ \Omega\times (0,T],
\end{align}
 and the diffusion equation
\begin{align}\label{eq:diffusion}
\partial_t u(x,t) = D u(x,t)\quad \mathrm{in}\ \Omega\times (0,T], \qquad Bu(x,t) = b(x)\quad \mathrm{on}\ \partial\Omega\times (0,T].
\end{align}
We use respectively the notations $\phi_t^f$ and $\phi_t^D$ to denote the exact flows of the subproblems~\eqref{eq:source} and~\eqref{eq:diffusion}.
The advantage of this subdivision is that equations~\eqref{eq:source} and~\eqref{eq:diffusion} can often be solved more efficiently than the main problem~\eqref{eq:parabolic}. The classical splitting method used to approximate the main problem~\eqref{eq:parabolic} is the Strang splitting. Starting from an arbitrary initial datum $u_n$, one step of the Strang splitting method, with a time step $\tau>0$, applied to equation~\eqref{eq:parabolic} is given by
\begin{align}\label{eq:StrangfDf}
u_{n+1}=\phi_\frac{\tau}{2}^f\circ\phi_\tau^D\circ\phi_\frac{\tau}{2}^f(u_n).
\end{align}
Interchanging the roles of the diffusion part and nonlinear part, it is also possible to define one step of the Strang splitting method as
\begin{align}\label{eq:StrangDfD}
u_{n+1}=\phi_\frac{\tau}{2}^D\circ\phi_\tau^f\circ\phi_\frac{\tau}{2}^D(u_n).
\end{align}
In both cases, we start the procedure with the initial condition $u_0$ of the parabolic problem~\eqref{eq:parabolic}.
Both methods~\eqref{eq:StrangfDf} and~\eqref{eq:StrangDfD} are formally of order of accuracy two. However, a reduction of order occurs in general, particularly for inhomogeneous boundary conditions, has observed in~\cite{Hun03} and~\cite{Hun95}. A suitable correction of the splitting algorithm has been proposed in~\cite{Ost16}, to avoid order reduction phenomenon. In~\cite{Ber20}, an alternate correction was proposed that depends only on the flow $\phi^f_{\frac{\tau}{2}}$ and facilitates the calculation of the correction. In this paper, we will however not use the techniques developed in~\cite{Ost16},~\cite{Ber20}. We prove that
when the Crank-Nicolson scheme is used to solve the diffusion equation~\eqref{eq:diffusion} in the splitting~\eqref{eq:StrangfDf}, there is no reduction of order away from a neighbourhood of $t=0$. This superconvergence property appears specific to the Crank-Nicolson scheme and when another Runge-Kutta method or even the exact flow itself is used to approximate the diffusion subproblem, the order reduction of the splitting~\eqref{eq:StrangfDf} is not avoided. We denote by $\phi_t^{D,CN}$ the numerical flow of the Crank-Nicolson method for the diffusion problem~\eqref{eq:diffusion}. We obtain the following splitting method, where $\phi^D_{\tau}$ has been replaced by $\phi_t^{D,CN}$ in the splitting~\eqref{eq:StrangfDf},
 \begin{align}\label{eq:StrangfDfcn}
u_{n+1}=\phi_\frac{\tau}{2}^f\circ\phi_\tau^{D,CN}\circ\phi_\frac{\tau}{2}^f(u_n).
\end{align}
We prove that the splitting~\eqref{eq:StrangfDfcn} has no order reduction when the nonlinearity $f=f(x)$ only depends on the space variable. More precisely, we prove in this case the following exact representation of the error at time $t_n=n\tau$,
\begin{align}\label{eq:globalerror}
u_n-u(t_n)=\left(r(\tau A)^n-\e^{n\tau A}\right)A^{-1}(Du_0+f),
\end{align}
where $A$ is the restriction of the operator $D$ to $\mathcal{D}(A)=\{u\in W^{2,p}(\Omega)\ ;\ Bu = 0 \ \mbox{on}\ \partial\Omega \}$, the set of functions satisfying the homogeneous boundary condition $Bu(x)=0$ on the boundary $\partial\Omega$, and where $r(z)=(1+\frac{z}{2})/(1-\frac{z}{2})$ is the stability function of the Crank-Nicolson scheme. 

As seen in~\cite[Theorem 4.2 and Theorem 4.4]{Lar91}, for the simplified case where $f$ and $b$ are both zeros in~\eqref{eq:parabolic}, that is for $\partial_t u(x,t)=Au(x,t)$, second order convergence results for A-stable methods (see~\cite[Chapter IV.3]{Hai10}) usually require $u_0\in\mathcal{D}(A^2)$. In contrast, L-stable methods are second order convergent outside a neighbourhood of the origin even if $u_0\in L^p(\Omega)$. The Crank-Nicolson scheme, although it is not L-stable but only A-stable, is second order convergent outside a neighbourhood of the origin for $u_0\in \mathcal{D}(A)$ (see~\cite[Theorem 2.1]{Han99}). Maximal parabolic regularity of A-stable Runge-Kutta methods is studied in~\cite{Lub16}. To the best of our knowledge, there exists no result in the literature which proves already that the Crank-Nicolson scheme is second order convergent outside a neighbourhood of the origin when applied to a nonlinear parabolic problem with inhomogeneous boundary conditions and an initial condition $u_0\in\{u\in W^{2,p}(\Omega)\ ;\ Bu = b \ \mbox{on}\ \partial\Omega \}$.

This is a direct consequence of the convergence result of the splitting~\eqref{eq:StrangfDfcn} since for $f=f(x)$, the splitting with Crank-Nicolson~\eqref{eq:StrangfDfcn} is equal to the Crank-Nicolson scheme~\eqref{eq:CN} applied to the whole problem~\eqref{eq:parabolic}.

We also provide numerical experiments in the case where $f$ depends on the solution $u$, in which case the splitting~\eqref{eq:StrangfDfcn} is different from the Crank-Nicolson scheme, and observe that the splitting method~\eqref{eq:StrangfDfcn} remains second order convergent in this case. Note also that the superconvergence property does not hold for the other splitting methods given by $u_{n+1}=\phi_\frac{\tau}{2}^{D,CN}\circ\phi_\tau^f\circ\phi_\frac{\tau}{2}^{D,CN}(u_n)$ nor does this splitting preserves stationary states for $f=f(x)$.

This paper is organized as follows. In Section~\ref{Section:Framework}, we describe an appropriate analytical framework for the analysis, where $D$ is chosen to be a second order elliptic operator and $B$ a first order differential operator corresponding to Dirichlet or oblique boundary conditions. In Section~\ref{Section:Method}, we describe precisely the algorithm corresponding to the splitting~\eqref{eq:StrangfDfcn}. In Section~\ref{Section:Analysis}, we restrict ourselves to the case where $f$ does not depend on the solution $u$. We first provide an exact representation of the local error (Proposition~\ref{prop}). We then prove formula~\eqref{eq:globalerror} for the global error (Theorem~\ref{thm:fx}) and additionally conclude that the stationary states are preserved by the splitting method~\eqref{eq:StrangfDfcn} (Remark~\ref{Cor}). In Section~\ref{Sec:num}, we provide numerical experiments to illustrate the properties of the splitting method~\eqref{eq:StrangfDfcn} compared to several natural splitting methods in dimensions one and two for constant and nonlinear terms. 
\section{Analytical framework}\label{Section:Framework}
We follow closely the framework and the notations given in the book~\cite[Chapter 3]{Lun95}.
Let $\Omega$ be an open bounded subset of $\mathbb{R}^d$ with $C^2$ boundary $\partial\Omega$ and dimension $d\geq 1$.
For $1<p<\infty$, let $D:W^{2,p}(\Omega)\rightarrow L^p(\Omega)$ be a second order  differential operator,
$$
D=\sum_{i,j=1}^{d}\frac{\partial}{\partial x_i}\left(a_{ij}(x)\frac{\partial}{\partial x_j}\right)+\sum_{i=1}^d b_i(x)\frac{\partial}{\partial x_i},
$$
where $a_{ij}$ and $b_i$ are real continuous functions on $\overline{\Omega}$. We assume that the matrix $[a_{ij}(x)]_{ij}$  is symmetric and uniformly positive definite on $\overline{\Omega}$, i.e for all $x\in \overline{\Omega}$ and for all $\xi\in\mathbb{R}^d$, $\xi^T [a_{ij}(x)]_{ij}\xi\geq c\xi^T \xi$, where $c>0$ is independent of $x$. The source term $f:L^p(\Omega)\rightarrow L^p(\Omega)$ is assumed continuously differentiable and we assume that the initial conditions $u_0$ belongs to $W^{2,p}(\Omega)$.

 The linear operator $B:W^{2,p}(\Omega)\rightarrow W^{1,p}(\Omega)$ is either defined for all $u\in W^{2,p}(\Omega)$ as $Bu=u$, which corresponds to Dirichlet boundary conditions or it is a first order differential operator defined for all $u\in W^{2,p}(\Omega)$ as 
$$
Bu(x)=\sum_{i=1}^d \beta_i(x)\frac{\partial u(x)}{\partial x_i}+\alpha(x)u(x),
$$
where $\beta_i$ and $\alpha$ are uniformly continuous and differentiable on $\overline{\Omega}$, which corresponds to Robin boundary conditions. We assume that $\alpha$ is not zero everywhere on $\partial\Omega$. The degenerate case, where $\alpha$ is the zero function, corresponding to Neumann boundary conditions, is discussed in Remark~\ref{rem:Neumann} below. If $B$ is a first order operator, we assume that the uniform nontangentiality condition is satisfied for all $x\in\partial\Omega$, 
$$
\left|\sum_{i=1}^d\beta_i(x)\vec{n}_i(x)\right|\geq c,
$$
where $c>0$ is independent of $x$ and where $\vec{n}(x)$ is the outwardly normal unit vector.  On the boundary $\partial\Omega$, $Bu|_{\partial\Omega}$ is the trace of $Bu\in W^{1,p}(\Omega)$ on $\partial\Omega$ and is therefore an element of $L^p(\partial\Omega)$. To avoid heavy notations, we simply write $Bu=b$, on $\partial\Omega$.
We assume that $b$ is a twice continuously differentiable function on the boundary $\partial\Omega$. To avoid stiffness of the solution or boundary layers at the initial time $t=0$, we assume in addition that the initial condition $u_0\in W^{2,p}(\Omega)$ satisfies the boundary conditions $Bu_0(x)=b(x)$ on $\partial\Omega$. 

The space $\{u\in W^{2,p}(\Omega)\ ; \ Bu=b\ \mbox{on}\ \partial\Omega \}$ is difficult to handle since it is not a linear subspace of $L^p(\Omega)$ if $b$ is not the zero function on $\partial \Omega$. Therefore, we provide a reformulation of the problem~\eqref{eq:parabolic} with homogeneous boundary conditions. We choose a function $z\in W^{2,p}(\Omega)$ which satisfies the boundary conditions $Bz=b$ on $\partial\Omega$. Such a function always exists with the assumptions that we made on $B,b$ and $\partial\Omega$. We define the function $\tilde{u}=u-z$, which satisfies the following differential problem with homogeneous boundary conditions,
\begin{align}\label{eq:parabolicZD}
\partial_t \tilde{u}(x,t) &= D\tilde{u}(x,t) + f(x,\tilde{u}(x,t)+z(x))+ Dz(x)\ \ \mathrm{in}\ \Omega\times (0,T],\nonumber\\
 \quad B\tilde{u}(x,t) &= 0\ \ \mathrm{on}\ \partial\Omega\times (0,T], \nonumber\\ 
 \tilde{u}(x,0)&=u_0(x)-z(x)\quad \mathrm{in}\ \Omega.
\end{align}
 We define the operator $(A,\mathcal{D}(A))$ as the restriction of the operator D to the domain $\mathcal{D}(A)=\{u\in W^{2,p}\ ;\ Bu=0\ \text{on } \partial \Omega\}$, i.e. $Au=Du$ for all 
$u\in \mathcal{D}(A)$. The operator $A$ therefore includes the homogeneous boundary conditions in its domain. Under the above assumptions, the operator $A$ is a closed densely defined linear operator satisfying the two following properties (see~\cite[Theorem 3.1.13]{Lun95} and~\cite[page 92]{Tho06}):
\begin{enumerate}
\item The resolvent set of $A$, $\rho(A)=\{\lambda\in\mathbb{C}\ ;\ \lambda I-A  \text{ is an isomorphism}\}$, contains the closure of the set $\Sigma_{\theta}=\{z\in \mathbb{C}\ ;\ z\neq 0,\ |\mathrm{arg}(z)|<\pi-\theta\}$, where $\theta\in (0,\frac{\pi}{2})$ is fixed,
\begin{equation}\label{Hyp:1}
\rho(A)\supset \overline{\Sigma}_{\theta}.
\end{equation}
\item For all $\lambda\in \Sigma_{\theta}$, the resolvent of $A$, $R(\lambda, A)=(\lambda I-A)^{-1}$, satisfies the following bound for the operator norm,
\begin{equation}\label{Hyp:2}
\|R(\lambda,A)\|\leq \frac{M}{|\lambda|},
\end{equation}
where $M\geq 1$.
\end{enumerate}
Note that, since $0\in\rho(A)$ by~\eqref{Hyp:1}, the operator $A$ is invertible and $A^{-1}$ is bounded. The operator $A$ is therefore the infinitesimal generator of an analytic uniformly bounded semigroup denoted $\e^{tA}$, given by
\begin{align*}
\e^{t A}=\frac{1}{2\pi i}\int_\Gamma \e^{zt}R(z,A)\mathrm{dz},
\end{align*}
where $\Gamma $ is the boundary of $\Sigma_{\theta}$ with imaginary part increasing along $\Gamma$ (see~\cite[Theorems 2.5.2 and 1.7.7]{Paz83}).

Since the homogeneous boundary conditions are included in the domain of $A$, we have the following reformulation of the problem~\eqref{eq:parabolicZD},
\begin{equation}\label{eq:parabolicZ}
 \partial_t\tilde{u}(t) = A\tilde{u}(t)+f(\tilde{u}(t)+z)+Dz,\quad \mbox{for}\quad t\in(0,T], \qquad \tilde{u}(0)=u_0-z,
\end{equation}
where we omit the variable $x$ in the notations, i.e. $\tilde{u}(t)$ denotes $\tilde{u}(x,t)$ and similarly for $z,u_0$, and $f$.
This equation has a solution $\tilde{u}\in C^1([0,T],L^p(\Omega))\cap C([0,T],\mathcal{D}(A))$ if $T$ is sufficiently small (see \cite[Proposition 7.1.10]{Lun95}), given by Duhamel formula,
\begin{align*}
\tilde{u}(t)&=\e^{t A}(u_0-z)+\int_0^t \e^{(t-s)A}(f(\tilde{u}(s)+z)+Dz)\dd.
\end{align*}
For $\tau>0$, denoting $t_n=n \tau$, since $\tilde{u}(t_n)\in \mathcal{D}(A)$, we have
\begin{align}\label{eq:exactSol}
u(t_{n}+\tau)&=z+\e^{\tau A}(u(t_n)-z)+\int_0^\tau \e^{(\tau-s)A}(f(u(t_n+s))+Dz)\dd.
\end{align}
Note that the above reformulation corresponds to the usual lifting
methodology to handle inhomogeneous boundary conditions.
A more general Banach space framework, that includes e.g. a bi-Laplacian
diffusion problem, will be discussed in Remark~\ref{rem:abstract}.

\begin{remark}\label{rem:Neumann}
Alternatively, on can consider pure Neumann boundary conditions, which corresponds to 
$$
B=\sum_{i=1}^d \beta_i(x)\frac{\partial}{\partial x_i}.
$$
In this particular case, we consider the operator
$$
D=\sum_{i,j=1}^{d}\frac{\partial}{\partial x_i}\left(a_{ij}(x)\frac{\partial}{\partial x_j}\right)+\sum_{i=1}^d b_i(x)\frac{\partial}{\partial x_i}+c(x),
$$
where $c(x)<c_{max}<0$ with $c(x)$ uniformly continuous. If all the additional assumptions on $D$ and $B$ are satisfied, then the operator $A$, defined as above, satisfied the assumptions~\eqref{Hyp:1} and~\eqref{Hyp:2} as required. Alternatively, if $c$ is the null function, we can also consider the subspace of functions with zero average on $\Omega$.
\end{remark}
\begin{remark}
Although we choose to present Theorem~\ref{thm:fx} in $L^p(\Omega)$ to simplify the presentation, it remains true for general complex separable Banach spaces and suitable analytic semigroups. In this case, the hypotheses on $A$ and $f$ are described in Section~\ref{Section:Analysis}, Remark~\ref{rem:abstract}.
\end{remark}
\section{The splitting method based on the Crank-Nicolson scheme}\label{Section:Method}
We describe precisely the algorithm for the splitting~\eqref{eq:StrangfDfcn} with a time step $\tau>0$. The same time step $\tau$ is used for the Strang splitting~\eqref{eq:StrangfDf} and for the Crank-Nicolson method used to approximate~\eqref{eq:diffusion}.
One step of the Crank-Nicolson scheme with a time step $\tau$ and an initial condition $u_0$ is given by the solution $u_1$ of the following equation,
\begin{equation}\label{eq:CN}
\frac{u_1(x)-u_0(x)}{\tau}=D\frac{u_1(x)+u_0(x)}{2}\quad \mbox{in } \Omega,\quad B\frac{u_0(x)+u_1(x)}{2}=b(x)\ \ \text{on }\ \partial\Omega.
\end{equation}
We denote by $u_1=\phi^{D,CN}_{\tau}(u_0)$ the solution of the problem~$\eqref{eq:CN}$.
Numerically, as explained in~\cite{Bes09} (see also Remark~\ref{rem:CNPrecision}), it is advantageous to save a linear system resolution and to define $v_1=\frac{u_1+u_0}{2}$ and to first find the solution $v_1$ of
\begin{equation}\label{eq:CNv}
2\frac{v_1(x)-u_0(x)}{\tau}=Dv_1(x)\quad \mbox{in } \Omega,\quad\quad Bv_1(x)=b(x)\ \ \text{on }\ \partial\Omega.
\end{equation}
One step of the Crank-Nicolson method is then given by
\begin{equation}\label{eq:2vu}
u_1=2v_1-u_0.
\end{equation}
One step of the splitting method~\eqref{eq:StrangfDfcn} is therefore given by the following algorithm.
\begin{alg}[Main algorithm for the splitting~\eqref{eq:StrangfDfcn}]\label{alg}
\
\begin{enumerate}
\item Given $u_n(x)$, compute  the solution $w(x,\frac{\tau}{2})$ of $\partial_t w(x,t) = f(x,w(x,t))$ in $\Omega$, $w(x,0)=u_n(x)$.
\item Compute the solution $v(x)$ of $(I-\frac{\tau}{2}D)v(x)=w(\frac{\tau}{2})$ in $\Omega$ with $ Bv(x)=b(x)$ on $\partial \Omega$. Compute $\hat{v}(x)=2v(x)-w(x,\frac{\tau}{2}).$
\item Compute  the solution $\hat{w}(x,\frac{\tau}{2})$ of $\partial_t \hat{w}(x,t) = f(x,\hat{w}(x,t))$ in $\Omega$, $\hat{w}(x,0)=\hat{v}(x)$.\\ Define $u_{n+1}(x)=\hat{w}(x,\frac{\tau}{2})$.
\end{enumerate}
\end{alg}
Note that one step of the Algorithm~\ref{alg} is given by the solution $u_{n+1}$ of the following problem,
\begin{align}
&\frac{\phi^f_{-\frac{\tau}{2}}(u_{n+1})-\phi^f_{\frac{\tau}{2}}(u_n)}{\tau}=D\frac{\phi^f_{-\frac{\tau}{2}}(u_{n+1})+\phi^f_{\frac{\tau}{2}}(u_n)}{2}\quad \mbox{in }\Omega,\nonumber\\
& B\frac{\phi^f_{-\frac{\tau}{2}}(u_{n+1})+\phi^f_{\frac{\tau}{2}}(u_n)}{2}=b\quad\mbox{on }\partial\Omega,\label{eq:Alg}
\end{align}
which can be solved using a linear solver for computing $\hat v =
\phi^f_{-\tau/2}(u_{n+1})$, then $u_{n+1}=\phi_{\tau/2}(\hat v)$, combined
for instance with a finite element discretization for the spatial
discretization.
\begin{remark}
Similarly to~\cite{Ost16}, the auxiliary function $z$ is only used as a tool to introduce homogeneous boundary conditions in the analysis. It is never used throughout the algorithm~\ref{alg} as seen in~\eqref{eq:Alg}.
\end{remark}
\section{Convergence analysis for a solution independent source term}\label{Section:Analysis}
In this section, we give an estimate of the error of the splitting~\eqref{eq:StrangfDfcn} when the source term $f=f(x)$ only depends on the space variable $x$. We assume $f\in L^p(\Omega)$, where again $1<p<\infty$. We can write the parabolic problem~\eqref{eq:parabolic} as
\begin{align}\label{eq:parabolicfx}
\partial_t u(x,t) &= Du(x,t) + f(x)\ \ \mathrm{in}\ \Omega\times (0, \infty),\nonumber\\
 \quad Bu(x,t) &= b(x)\ \ \mathrm{on}\ \partial\Omega\times (0,\infty), \nonumber\\ 
 u(x,0)&=u_0(x)\quad \mathrm{in}\ \Omega,
\end{align}
Since we are interested in the semi-discretization in time, for brevity of notations, we write, from now on, $u(t)$ and $f$ instead of $u(x,t)$ and $f(x)$.
We denote by $r(y)$ the stability function of the Crank-Nicolson method given by, 
\begin{equation}\label{eq:rationalCN}
r(y)=\frac{1+\frac{y}{2}}{1-\frac{y}{2}}.
\end{equation}
We recall that, for any Runge-Kutta method applied with a time step $\tau>0$ to the Dahlquist scalar test equation (see~\cite[Definition IV.2.1]{Hai10})
$$\frac{dx}{dt}(t)=\lambda x(t),\quad x(0)=x_0,$$ with $\lambda\in \mathbb{C}$, one obtains the induction $x_{n+1}=R(h\lambda)x_n$, where $R:\mathbb{C}\rightarrow \mathbb{C}$ is a rational approximation of the exponential that we call the stability function of the Runge-Kutta method. For a fixed degree of the numerator and denominator, the rational approximations that have the highest order of approximation are called the Padé approximations of the exponential (cf. \cite[Chapter IV.3]{Hai10}) and efficient Runge-Kutta methods are typically constructed to have a stability function equal to a Padé approximation. The approximation $r(z)$ that corresponds to the Crank-Nicolson stability function is the (1,1)-Padé approximation~\eqref{eq:rationalCN}.   
An A-stable method is by definition a Runge-Kutta method whose stability function verifies $|R(y)|\leq 1$ for all $y\in \mathbb{C}^-=\{z\in \mathbb{C}\ ;\ \Re(z)\leq 0\}$. We recall that the only Padé approximations that verify this property are the $(j,k)$-Padé approximations with $k\leq j\leq k+2$ (see~\cite[Theorem 4.12]{Hai10}).
\begin{remark}\label{rem:CNonly}
Consider the following ordinary differential equation 
$$\frac{d x}{dt}(t)=\lambda x(t)+b,\quad x(0)=x_0,$$
with $\lambda,b\in\mathbb{C}$.
 A Runge-Kutta method with stability function $R(y)$ yields
\begin{equation}\label{eq:RK}
x_{n+1}=R(h\lambda)x_n+\frac{R(h\lambda)-1}{h\lambda}b.
\end{equation}
The Strang splitting $x_{n+1}=\phi_\frac{h}{2}^b\circ\phi_h^{\lambda,RK}\circ\phi_\frac{h}{2}^b(x_n),$
where $\phi_\frac{h}{2}^b(x)=x+\frac{h}{2}b$ and 
$\phi_h^{\lambda,RK}(x)=R(h\lambda)x$ yields 
\begin{equation}\label{eq:StrangRK}
x_{n+1}=R(h \lambda)(x_n+\frac{h}{2}b)+\frac{h}{2}b.
\end{equation}
Equations~\eqref{eq:RK} and~\eqref{eq:StrangRK} coincide if and only if 
$$
R(y)=r(y)=\frac{1+\frac{y}{2}}{1-\frac{y}{2}},
$$
which is the stability function~\eqref{eq:rationalCN} of the Crank-Nicolson scheme. Therefore, the only Runge-Kutta methods for which equations~\eqref{eq:RK} and~\eqref{eq:StrangRK} coincide are the ones with a stability function equal to $r(y)$.
\end{remark}
\begin{remark}\label{rem:CNequiv}
One step of the Crank-Nicolson scheme applied to the parabolic problem~\eqref{eq:parabolic} is given by the solution $u_{n+1}$ of the following equation, 
\begin{align}
&\frac{u_{n+1}-u_n}{\tau}=D\frac{u_{n+1}+u_n}{2}+\frac{f(u_{n+1})+f(u_n)}{2}\quad \mbox{in }\Omega,\nonumber\\
& B\frac{u_{n+1}+u_n}{2}=b\quad\mbox{on }\partial\Omega.\label{eq:CNfu}
\end{align}
We observe that, for $f=f(x)$, formula~\eqref{eq:CNfu} is equal to formula~\eqref{eq:Alg} which describes one step of the splitting~\eqref{eq:StrangfDfcn}. Therefore, we deduce that for $f=f(x)$, the Strang splitting with Crank-Nicolson~\eqref{eq:StrangfDfcn} is equivalent to the Crank-Nicolson scheme itself applied to the whole problem~\eqref{eq:parabolicfx}.
\end{remark}
\subsection{Main results}
The main result of this paper is the following theorem, which states that the splitting~\eqref{eq:StrangfDfcn} yields a method of order $2$ of accuracy away form a neighbourhood of the origin $t=0$ on unbounded intervals ($t>0$). In contrast, in a neighbourhood of zero, the order of accuracy reduces to one.
\begin{theorem}\label{thm:fx}
Let $e_n=u_n-u(t_n)$ be the global error of the splitting method~\eqref{eq:StrangfDfcn} applied to the parabolic problem~\eqref{eq:parabolicfx}, where $u_0\in W^{2,p}(\Omega)$ and satisfies $Bu_0=b$ on $\partial\Omega$.
Then, for all $n\geq 0$, the global error $e_n$ is given by
\begin{align}\label{eq:global}
e_n=\left(r(\tau A)^n-\e^{n\tau A}\right)A^{-1}(Du_0+f)
\end{align}
and it satisfies the bound
\begin{align*}
\|e_n\|_{L^p(\Omega)}\leq \frac{C\tau^2}{t_n},
\end{align*}
where $C$ is a constant independent of $\tau$, $n$, and $t_n=n\tau$. 
\end{theorem}
\begin{remark}
The estimate of Theorem~\ref{thm:fx} could be used to derive a fully discrete estimate of the form
$$
\|u_n^h-u(t_n)\| \leq C (\tau^2/t_n + h^p)
$$
where $u_n^h$ denotes a standard finite element discretization of order $p$ on a spatial mesh with size $h$. The idea of the proof is to rely on the triangle inequality
$$
\|u_n^h-u(t_n)\| \leq C \|u_n^h-u^h(t_n)\| + \|u^h(t_n)-u(t_n)\|
$$
where $u^h(t_n)$ denotes the semi-discretization in space at time $t_n$. Then, observe that the estimate of Theorem 4.1 also holds for $u_n^h-u^h(t_n)$ uniformly with respect to the spatial mesh size $h$, using that the space discretization of the diffusion operator A is a self-adjoint operator that satisfies assumptions analogous to~\eqref{Hyp:1} and~\eqref{Hyp:2}, see e.g.~\cite{CroPoly}.
\end{remark}
\begin{remark}\label{rem:abstract}
One can extend to more abstract problems the framework described in Section 2 and consider for example, a problem where $D$ is the bi-Laplacian with appropriate boundary conditions. This formulation also include Galerkin approximation of parabolic problems (see~\cite{Tho06}). More generally, for a complex separable Banach space $X$, we need $A:\mathcal{D}(A)\rightarrow X$, to be a closed densely defined linear operator satisfying the conditions~\eqref{Hyp:1} and~\eqref{Hyp:2} or equivalently we require $A$ to be the infinitesimal generator of a uniformly bounded analytic semigroup (~\cite[Theorems 2.5.2]{Paz83}). We require $f:X\rightarrow X$ to be continuously differentiable. The operator $D:\mathcal{D}(D)\rightarrow X$ is assumed to be an extension of $A$, i.e. $\mathcal{D}(A)\subset\mathcal{D}(D)$. We require $z\in\mathcal{D}(D)$  and $u_0-z\in \mathcal{D}(A)$. Then, the problem~\eqref{eq:parabolicZ} has a unique solution $\tilde{u}\in C^1([0,T],X)\cap C([0,T],\mathcal{D}(A))$ if $T$ is sufficiently small (see \cite[Proposition 7.1.10]{Lun95}). The main problem~\eqref{eq:parabolic}, for $u=\tilde{u}+z$ becomes
\begin{equation*}
 \partial_tu(t) = Du(t)+f(u(t)),\ \ \mbox{for}\ \ t\in(0,T],\quad u(t)-z\in\mathcal{D}(A)\ \ \mbox{for}\ \ t\in(0,T], \quad u(0)=u_0.
\end{equation*}
For this problem, the splitting~\eqref{eq:StrangfDfcn}, is defined as above for~\eqref{eq:parabolic}, where the boundary conditions of the problem~\eqref{eq:parabolic} have been replaced by $u(t)-z\in\mathcal{D}(A)$. Under all those conditions, for $f=f(x)$, the convergence analysis (Theorem~\ref{thm:fx}) remains true, i.e. the splitting method~\eqref{eq:StrangfDfcn} has no reduction of order away from the origin and the global error is given by formula~\eqref{eq:global}.
\end{remark}
\begin{remark}\label{Cor}
Since the splitting method~\eqref{eq:StrangfDfcn} applied to the parabolic problem~\eqref{eq:parabolicfx} is equivalent to the Crank-Nicolson scheme, it must preserve exactly the stationary states. Precisely, for $Du_0+f=0$, the error satisfies $e_n=0$. 
For general nonlinearities that depend on the solution $f=f(u)$, the stationary states are not preserved. This is not surprising since in~\cite{Mcl16}, it is proved that a method that is not a Butcher-series method, like the splitting methods, and which is invariant by an affine change of variable (precisely affine-equivalent), cannot preserve all stationary states. 
\end{remark}
\subsection{Preliminaries}
We give some basic properties of the rational approximation $r(y)$ in~\eqref{eq:rationalCN}. The following properties are obtained with straightforward computations.
\begin{lemma}\label{lemma:1}
We have the following formulas,
\begin{equation}\label{eq:rp1}
r(y)+1=\frac{2}{1-\frac{y}{2}},\qquad r(y)-1=\frac{y}{1-\frac{y}{2}},
\end{equation}
and 
\begin{align}\label{eq:rme}
 \left(r(y)-\e^{y}\right)&=\frac{y}{2}(r(y)+1)-(\e^{y}-1).
\end{align}
\end{lemma}
The rational approximation $r(y)$ satisfies the following result, in the context of homogeneous parabolic problems, which is proved in the Appendix~\ref{Appendix:A} (see~\cite[Theorem 2.1]{Han99} for a proof in a more general case).
\begin{theorem}\label{thm:hansbo}
For $u_0\in \mathcal{D}(A)$, where $A$ satisfies the assumptions of Section~\ref{Section:Framework} and $r(y)$ is defined in~\eqref{eq:rationalCN}, we have the following error estimate,
$$
\|(r(\tau A)^n-\e^{\tau n A})u_0\|_{L^p(\Omega)}\leq \frac{C \tau^2}{t_n}\|Au_0\|_{L^p(\Omega)},
$$ 
where $C$ is a constant independent of $u_0$, $\tau$, $n$ and $t_n=n\tau$. 
\end{theorem}
Note that this corresponds to an estimate of the splitting~\eqref{eq:StrangfDfcn} for the specific case where the problem is homogeneous, i.e. $f=0$, and with homogeneous boundary conditions. 
In what follows, we give an exact representation of the numerical solution of the splitting~\eqref{eq:StrangfDfcn} in term of the operator $A$. We recover homogeneous boundary conditions with an appropriate change of variable. Let $z\in W^{2,p}(\Omega)$ be the same function chosen in~\eqref{eq:exactSol} and satisfying $Bz=b$ on $\partial\Omega$. Defining $\tilde{v}_1=v_1-z$, we rewrite equation~\eqref{eq:CNv} as follows,
$$
2\frac{\tilde{v}_1-u_0}{\tau}+\frac{2}{\tau}z=A\tilde{v}_1+Dz,
$$
where we recall that the homogeneous boundary conditions are included in the domain of $A$.
This gives the following expression for $\tilde{v}_1$,
$$
\tilde{v}_1=\left(I-\frac{\tau}{2}A\right)^{-1}u_0+\frac{\tau}{2}\left(I-\frac{\tau}{2}A\right)^{-1}Dz-\left(I-\frac{\tau}{2}A\right)^{-1}z.
$$
We denote $u_1=\phi^{D,CN}_{\tau}(u_0)$, the solution of the problem~\eqref{eq:CN}. Therefore, from equation~\eqref{eq:2vu}, we have the following formula,
\begin{align*}
u_1=\phi^{D,CN}_\tau(u_0)&=2\tilde{v}_1-u_0+2z\\
&=\left(2\left(I-\frac{\tau}{2}A\right)^{-1}-I\right)(u_0-z)+\tau\left(I-\frac{\tau}{2}A\right)^{-1}Dz+z.
\end{align*}
Using the equalities~\eqref{eq:rp1}, we obtain
\begin{equation}\label{eq:SolCN}
\phi^{D,CN}_{\tau}(u_0)=z+r(\tau A)(u_0-z)+(r(\tau A)-I)A^{-1}Dz.
\end{equation}
One step of the splitting~\eqref{eq:StrangfDfcn}, is denoted by $\mathcal{S}_{\tau}$, i.e.
\begin{equation*}
u_{n+1}=\mathcal{S}_{\tau}(u_n):=\phi_\frac{\tau}{2}^f\circ\phi^{D,CN}_{\tau}\circ\phi_\frac{\tau}{2}^f(u_n).
\end{equation*}
Using the following representation of the exact flow $\phi_t^f$ for $f=f(x)$,
\begin{equation}\label{eq:SolSource}
\phi^f_{t}(u_0)=u_0+tf,
\end{equation}
and formula~\eqref{eq:SolCN} for $\phi^{D,CN}_{\tau}$, we obtain the following expression for $\mathcal{S}_{\tau}(u_n)$:
\begin{align}\label{eq:numerical}
\mathcal{S}_{\tau}(u_n)&=z+r(\tau A)(u_n-z)+(r(\tau A)-I)A^{-1}Dz+\frac{\tau}{2}(r(\tau A)+I)f.
\end{align}
\begin{remark}\label{Rem:RK}
Take any $A$-stable Runge-Kutta method and denote by $R(z)$ its stability function.  Then if this Runge-Kutta method is used to solve the diffusion equation~\eqref{eq:diffusion}, instead of the Crank-Nicolson method, one obtains formula~\eqref{eq:numerical} for the numerical solution with $r(\tau A)$ replaced with $R(\tau A)$. Indeed, let $u$ be the solution of the diffusion equation~\eqref{eq:diffusion}. Let $z\in W^{2,p}(\Omega)$ be a function satisfying $Bz=b$ on $\partial\Omega$. Then, we have $\partial_t u=A(u-z+A^{-1}Dz)$ for  $t>0$. Defining $y=u-z+A^{-1}Dz$, we obtain the following equivalent problem,
$$
\partial_t y = Ay\  \ \mbox{for}\ t>0, \quad y(0)=u_0-z+A^{-1}Dz.
$$
Applying one step of the Runge-Kutta method with initial condition $y(0)$ and with a time step $\tau$ gives,
$$
y_1=R(\tau A)(u_0-z+A^{-1}Dz).
$$
Hence, we obtain the same formula~\eqref{eq:SolCN} with $r(\tau A)$ replaced with $R(\tau A)$ for the numerical solution of the diffusion~\eqref{eq:diffusion},
$$
u_1=R(\tau A)(u_0-z+A^{-1}Dz)+z-A^{-1}Dz=z+R(\tau A)(u_0-z)+(R(\tau A)-1)A^{-1}Dz.
$$
\end{remark}
\subsection{Local error}
We start with the following proposition that gives an exact representation of the local error.
\begin{proposition}\label{prop}
The local error $\delta_{n+1}=\mathcal{S}_{\tau}(u(t_n))-u(t_{n+1})$ of the splitting~\eqref{eq:StrangfDfcn} satisfies the following identity,
\begin{equation*}
\delta_{n+1}=(r(\tau A)-\e^{\tau A})A^{-1}\e^{t_n A}(Du_0+f).
\end{equation*}
\end{proposition}
\begin{proof}
We have the following representation of $\delta_{n+1}$:
\begin{align}
\delta_{n+1}&=(r(\tau A)-\e^{\tau A})(u(t_n)-z)+(r(\tau A)-\e^{\tau A})A^{-1}Dz\nonumber\\
&+\frac{\tau}{2}(r(\tau A)+I)f-A^{-1}(\e^{\tau A}-I)f.\label{eq:LocalError}
\end{align}
From formula~\eqref{eq:rme}, we obtain,
\begin{equation}\label{eq:LocalError2}
\delta_{n+1}=(r(\tau A)-\e^{\tau A})(u(t_n)-z+A^{-1}Dz+A^{-1}f).
\end{equation}
From equation~\eqref{eq:parabolicZ}, we know that $\tilde{u}(t_n)$ is the solution of the following problem, 
\begin{align*}\label{eq:LocalErrorFormula1}
\partial_t \tilde{u}(t) = A \tilde{u}(t) + f+Dz\quad \mathrm{for}\ t\in(0,T] , \qquad \tilde{u}(0)=u_0-z.
\end{align*}
Hence, $\delta_{n+1}$ satisfies 
\begin{equation*}\label{eq:LocalErrorFormula2}
\delta_{n+1}=(r(\tau A)-\e^{\tau A})A^{-1}\partial_t \tilde{u}(t_n).
\end{equation*}
Using the variation of constant formula, we obtain,
\begin{align*}
\partial_t \tilde{u}(t_n) &= A \tilde{u}(t_n) + f+Dz\\
&=A \e^{t_n A}\tilde{u}_0+A\int_0^{t_n}\e^{(t_n-s)A}(f+Dz)\mathrm{ds} + f+Dz\\
&=A \e^{t_n A}\tilde{u}_0+(\e^{t_nA}-I)(f+Dz) + f+Dz\\
&=\e^{t_n A}(A\tilde{u}_0+f+Dz)\\
&=\e^{t_n A}(Du_0+f)
\end{align*}
where we use $A\tilde{u}_0+Dz=D(\tilde{u}_0+z)=Du_0$.
This concludes the proof.
\end{proof}
\begin{remark}
Assume that another Runge-Kutta method, with stability function $R(y)$, is used to solve the diffusion equation~\eqref{eq:diffusion} involved in the splitting method~\eqref{eq:StrangfDfcn}. Then, from Remark~\ref{Rem:RK}, we observe that the local error $\delta_{n+1}$ still satisfies formula~\eqref{eq:LocalError} with $r(\tau A)$ replaced with $R(\tau A)$.
However, the representation~\eqref{eq:LocalError2} is specific to Runge-Kutta methods having $r(y)$ as a stability function. Indeed, to find this representation of the local error, we used formula~\eqref{eq:rme} which is a property satisfied only by the (1,1)-Padé approximation~\eqref{eq:rationalCN}.   
\end{remark}
\subsection{Global error}
We are now in position to prove Theorem~\ref{thm:fx} for the global error $e_n=u_n-u(t_n)$ of the splitting~\eqref{eq:StrangfDfcn}. We observe that 
$$
e_{n+1}=\mathcal{S}_{\tau}(u_n)-\mathcal{S}_{\tau}u(t_n)+\mathcal{S}_{\tau}u(t_n)-u(t_{n+1})=\mathcal{S}_{\tau}(u_n)-\mathcal{S}_{\tau}u(t_n)+\delta_{n+1}.
$$
\begin{proof}[Proof of Theorem~\ref{thm:fx}]
From formula~\eqref{eq:numerical} for $\mathcal{S}_{\tau}(u_n)$, we obtain
\begin{align*}
\mathcal{S}_{\tau}(u_n)-\mathcal{S}_{\tau}u(t_n)&=r(\tau A)(u_n-u(t_n)).
\end{align*}
Therefore, we deduce
\begin{align*}
e_{n+1}&=r(\tau A) e_n +\delta_{n+1}.
\end{align*}
Hence, since $\e_0=0$, and using Proposition~\ref{prop} for the local error, we obtain
\begin{align*}
e_n&=\sum_{k=0}^{n-1}r(\tau A)^{n-k-1}\delta_{k+1}\\
&=\bigg((r(\tau A)-\e^{\tau A})\sum_{k=0}^{n-1}r(\tau A)^{n-k-1}\e^{k \tau A}\bigg)A^{-1}(Du_0+f)\\
&=\left(r(\tau A)^n-\e^{n\tau A}\right)A^{-1}(Du_0+f).
\end{align*}
Therefore, using that $A$ satisfies the hypotheses of Theorem~\ref{thm:hansbo}, and since $A^{-1}(Du_0+f)\in \mathcal{D}(A)$, we obtain the following estimate,
\begin{align*}
\|e_n\|_{L^p(\Omega)}\leq \frac{C}{t_n}\tau^2\|Du_0+f\|_{L^p(\Omega)},
\end{align*}
which concludes the proof of Theorem~\ref{thm:fx}.
\end{proof}
\begin{remark}
Formula~\eqref{eq:global} can also be obtained directly with an appropriate change of variable. This alternative proof makes clear why the global error $e_n$ can be expressed as a difference between the rational function $r(\tau A)^n$ and the semigroup $\e^{n\tau A}$.
Indeed by the affine change of variable $\hat{u}(x,t)=u(x,t)-\hat{z}(x)$,
where $\hat{z}(x)=z(x)-A^{-1}Dz(x)-A^{-1}f(x)$, we obtain $\partial_t \hat{u}(x,t)=A\tilde{u}(x,t)$.
Therefore, we can rewrite equation~\eqref{eq:parabolicfx} as follows
\begin{align*}
\partial_t \hat{u}(x,t) &= A\hat{u}(x,t)\ \ \mathrm{in}\ \Omega\times (0, \infty),\nonumber\\
 \quad B\hat{u}(x,t) &= 0\ \ \mathrm{on}\ \partial\Omega\times (0,\infty), \nonumber\\ 
 \hat{u}(x,0)&=u_0(x)-\hat{z}(x)\quad \mathrm{in}\ \Omega,
\end{align*}
whose solution is given by $\hat{u}(x,t)=\e^{tA}A^{-1}(Du_0(x)+f(x))$.
Hence, we have the following formula for the exact solution, 
\begin{align}\label{eq:exact}
u(x,t)=\e^{tA}A^{-1}(Du_0(x)+f(x))+\hat{z}(x).
\end{align}
By Remark~\ref{rem:CNequiv}, we know that the numerical solution $u_n$ of the splitting~\eqref{eq:StrangfDfcn} is given by $n$ iterations of the Crank-Nicolson method applied to the whole problem~\eqref{eq:parabolicfx}. Since Runge-Kutta methods are affine invariant, we obtain the following formula for $u_n(x)$,
\begin{equation}\label{eq:unRK}
u_n(x)=r(\tau A)^n A^{-1}(Du_{0}(x)+f(x))+\hat{z}(x).
\end{equation}
Subtracting~\eqref{eq:unRK} and~\eqref{eq:exact} at time $t_n$, we therefore obtain formula~\eqref{eq:global} for the global error.
Note that for any A-stable Runge-Kutta method with stability function $R(y)$, applied to the whole problem~\eqref{eq:parabolicfx}, the numerical solution $u_n$ and the error $e_n$ are given by formulas~\eqref{eq:unRK} and~\eqref{eq:global} with $r(\tau A)$ replaced with $R(\tau A)$.
\end{remark}
\section{Numerical experiments}\label{Sec:num}
We first describe the parameters and the notations that we use for the numerical experiments that follow.

 For the one dimensional problems, we choose $N=1000$ uniform grid points to discretize the domain $\Omega=(0,1)$, i.e. the mesh is of size $h=\frac{1}{N}=10^{-3}$. We denote by $U_{n,l}$ and $U_l(t_n)$, the approximations of $u_n(x_l)$ and $u(t_n,x_l)$, where $x_l=lh$, i.e. $U_n$ and $U(t_n)$ are vectors in $\mathbb{R}^N$. For the two dimensional problems, we discretize the domain $\Omega=(0,1)^2$ with a uniform mesh of size $h=10^{-2}$. We denote by $U_{n,l,m}$ and $U_{l,m}(t_n)$  the approximations of $u_n(x_l,y_m)$ and $u(t_n,x_l,y_m)$, where $y_m=mh$. The operators $\partial_{xx}$ and $\partial_{xx}+\partial_{yy}$ are approximated with the standard second order finite difference approximation and ghost points are used for the normal derivatives in the boundary conditions.
The final time is $T=0.1$. We apply all considered splitting methods with the time steps $\tau=0.02\cdot2^{-k}$ for $k=0,\ldots,6$.

 In the splitting algorithms, the source term equation~\eqref{eq:source} is solved exactly. In the splitting~\eqref{eq:StrangfDf} and~\eqref{eq:StrangDfD}, when we say that the diffusion equation is solved exactly, we mean that it is solved with a Krylov based algorithm developed in~\cite{Nie12}, with a tolerance close to machine precision. The reference solutions are computed with the Crank-Nicolson method (for $d=1$) and the classical four stage Runge-Kutta method (for $d=2$) with a small time step $\tau=0.02\cdot2^{-10}\approx 2\cdot 10^{-5}$.
 
The splitting~\eqref{eq:StrangfDf} using the exact flow of the diffusion part is denoted \emph{StrangEXP} and the splitting~\eqref{eq:StrangfDfcn} is denoted \emph{StrangCN}. 
We also consider the splitting methods \emph{StrangGauss}, \emph{StrangRadau} and \emph{StrangLobatto}, which are constructed similarly to the splitting method~\emph{StrangCN}, but where the diffusion problem~\eqref{eq:diffusion} is approximated  with the two stage Gauss method (order 4), the two stage Radau 1a method (order 3) and the two stage Lobatto 3c method (order 2) (see the Runge-Kutta Butcher tableau and stability functions in Appendix~\ref{Appendix:B}).
Similarly, the splitting~\eqref{eq:StrangDfD} is denoted by \emph{StrangEXP2}. We denote by \emph{StrangCN2}, \emph{StrangGauss2}, \emph{StrangRadau2} and \emph{StrangLobatto2}, the splitting methods corresponding to~\eqref{eq:StrangDfD}, where the diffusion equation~\eqref{eq:diffusion} is approximated with one of the methods described above.
 
The error of a splitting method at time $t_k=k\tau$ is defined as $u_k-u(k\tau)$, where $u(k\tau)$ is given by the reference solution at time $k\tau$. In the numerical experiments we always estimate the error with the trapezoidal approximation of the $L^2(\Omega)$ norm at final time $T=n\tau$ (except in Figure~\ref{Fig:ExvsCNb}). In dimension one, the estimate of the $L^2(0,1)$ error $E_k$ at time $t_k$ is given by
$$
E^2_k:=h\sum_{l=1}^{N-1} \frac{|U_{k,l}-U_l(t_k)|^2+|U_{k,l+1}-U_{l+1}(t_k)|^2}{2}.
$$
In dimension two, the $L^2((0,1)^2)$ error $E_k$ is approximated similarly,
\begin{align*}
E_k^2:&= h^2\sum_{l,m=1}^{N-1} \frac{|U_{k,l,m}-U_{l,m}(t_k)|^2+|U_{k,l+1,m}-U_{l+1,m}(t_k)|^2}{4}\\
&+\frac{|U_{k,l,m+1}-U_{l,m+1}(t_k)|^2+|U_{k,l+1,m+1}-U_{l+1,m+1}(t_k)|^2}{4}.
\end{align*}
In Figure~\ref{Fig:ExvsCNb}, we consider additional norms to estimate the error.  We consider the approximation of the $L^{\infty}([0,T],L^2(0,1))$ norm of the error,  
\begin{equation}\label{eq:norm1}
E_{\infty,0}:=\max_{k=0,\ldots, n}E_k.
\end{equation}
Similarly, we consider the approximation of the $L^{\infty}([0.02,T],L^2(\Omega))$ norm of the error,
\begin{equation}\label{eq:norm2}
E_{\infty,0.02}:=\max_{k=\frac{0.02}{\tau},\ldots, n}E_k.
\end{equation}
Another estimate of the error is provided where we compute an approximation of the $L^{\infty}([0,T],L^2(0,1))$ norm of time multiplied by the error, precisely
\begin{equation}\label{eq:norm3}
\hat{E}_{\infty,0}:=\max_{k=0,\ldots, n}\|t_kE_k\|_2.
\end{equation}
\begin{remark}\label{rem:CNPrecision}
In Figure~\ref{Fig:Stationary}, we implement the Crank-Nicolson method using the standard implementation of Runge-Kutta methods to avoid rounding error. To solve $\frac{dx}{dt}(t)=Ax(t)$ with $x(0)=x_0$, we start to resolve the linear system
$
k_1=Ax_0$ and
$
k_2=Ax_0+\frac{ \tau}{2}A(k_1+k_2).
$
Then, we write
$
x_1=x_0+\frac{\tau}{2}k_1+\frac{\tau}{2}k_2,
$
and similarly for the two stage Gauss method. For even higher accuracy, one should use in addition a compensated summation algorithm (see~\cite[Algorithm VIII.5.1]{hlw10}). 
\end{remark}
\subsection{Solution independent source term}
In the first series of experiments, we consider the following parabolic problem on $\Omega=(0,1)$ with $t\in [0,T]$, with Dirichlet boundary conditions, which satisfies the hypotheses of Theorem~\ref{thm:fx},
\begin{align}\label{eq:parabolic1}
\partial_t u(x,t) = \partial_{xx} u(x,t) + 1\quad \mathrm{in}\ (0,1)\times (0,T], \qquad u(0,t)=u(1,t)=1, \qquad u(x,0)=1.
\end{align}

\begin{figure}
\centering
\begin{subfigure}{.5\textwidth}
\centering
\includegraphics[width=\textwidth]{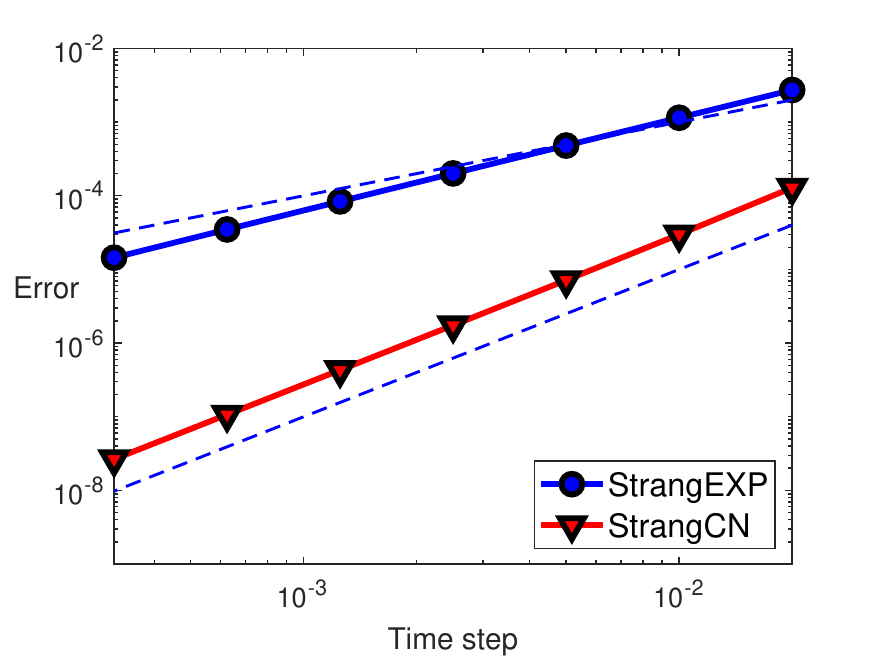} 
\caption{\textit{\emph{StrangCN} has no order reduction.}}\label{Fig:ExvsCNa}
\end{subfigure}%
\begin{subfigure}{.5\textwidth}
\centering
\includegraphics[width=\textwidth]{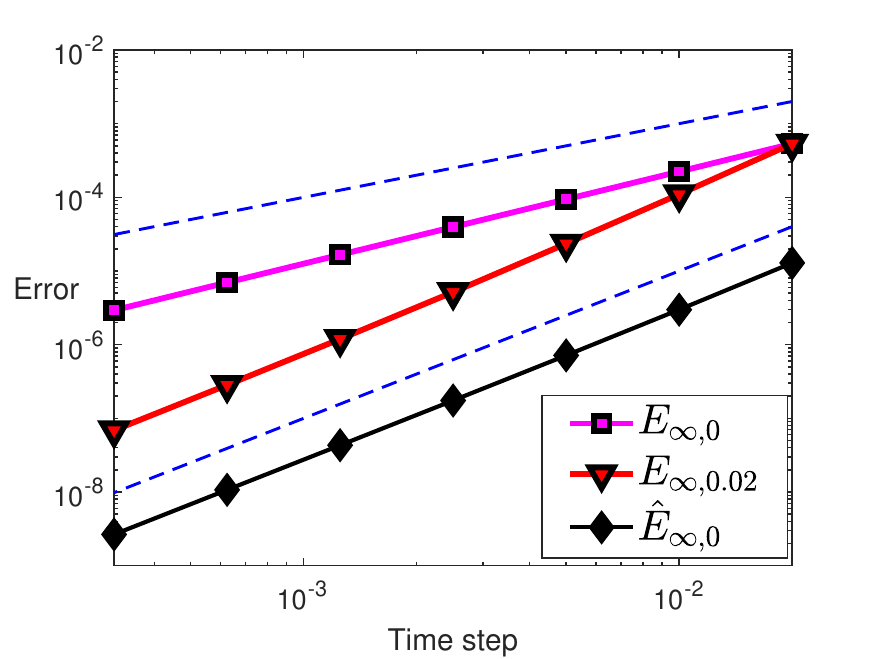}
\caption{\textit{\emph{StrangCN} is not of order two near $t=0$.}}\label{Fig:ExvsCNb}
\end{subfigure}
\caption{\textit{Solving the diffusion part~\eqref{eq:diffusion} with the Crank-Nicolson scheme allows to remove the reduction of order of the Strang splitting method at final time $T=0.1$, when applied to the 1d problem $\partial_t u=\partial_{xx} u+1$ with inhomogeneous Dirichlet boundary conditions. However, in a neighbourhood of $t=0$, the reduction of order is not avoided. Reference slopes one and two are given in dashed lines.}}\label{Fig:ExvsCN}
\end{figure}

In Figure~\ref{Fig:ExvsCNa}, we compare the splitting \emph{StrangEXP} and the splitting \emph{StrangCN} for the problem~\eqref{eq:parabolic1}. This simple example illustrates the superiority of the splitting~\eqref{eq:StrangfDfcn} compared to the splitting~\eqref{eq:StrangfDf}. Indeed, the splitting method~\eqref{eq:StrangfDfcn} avoids order reduction contrary to the Strang splitting~\eqref{eq:StrangfDf} and it allows a high gain of accuracy for no additional computational cost. Note that in the non-generic case where the source term satisfies the condition $Bf(x)=0$ on the boundary, there is no order reduction since no perturbation of the boundary occurs when solving the source equation~\eqref{eq:source} in the splitting~\eqref{eq:StrangfDf}. For example, for the problem $\partial_t u(x,t)=\partial_{xx}u(x,t) +\e^{-x}$ with $u(0,t)-\partial_n u(0,t)=1$, $u(1,t)+\partial_n u(1,t)=1$, and $u(x,0)=1$, the splitting~\eqref{eq:StrangfDf} is of order two and its convergence curve is superposed to the one of the splitting~\eqref{eq:StrangfDfcn}. The convergence curves are not drawn for conciseness.
  
In Figure~\ref{Fig:ExvsCNb}, we estimate the error of the splitting \emph{StrangCN}, applied to the problem~\eqref{eq:parabolic1} with the norm $L^{\infty}([0,T],L^2(0,1))$ and the norm $L^{\infty}([0,T],L^2(0.02,1))$, i.e. we compare $E_{\infty,0}$ and $E_{\infty,0.02}$, given by~\eqref{eq:norm1} and~\eqref{eq:norm2}. We observe that, $E_{\infty,0}$ does not decrease quadratically with respect to the time step $\tau$. This is expected since the bound of the error of the splitting~\eqref{eq:StrangfDfcn} given in Theorem~\ref{thm:fx} has order reduction down to one in a neighbourhood of $t=0$. In comparison, with $E_{\infty,0.02}$, which avoids a neighbourhood of $t=0$, we recover the second order convergence. We also observe that $\hat{E}_{\infty,0}$, given by formula~\eqref{eq:norm3}, decays quadratically with respect to $\tau$. This suggests that the error estimate $\bigo(\frac{\tau^2}{t_n})$ of Theorem~\ref{thm:fx} is optimal in a neighbourhood of $t=0$.
\begin{figure}
\centering
\begin{subfigure}{.5\textwidth}
\centering
\includegraphics[width=\textwidth]{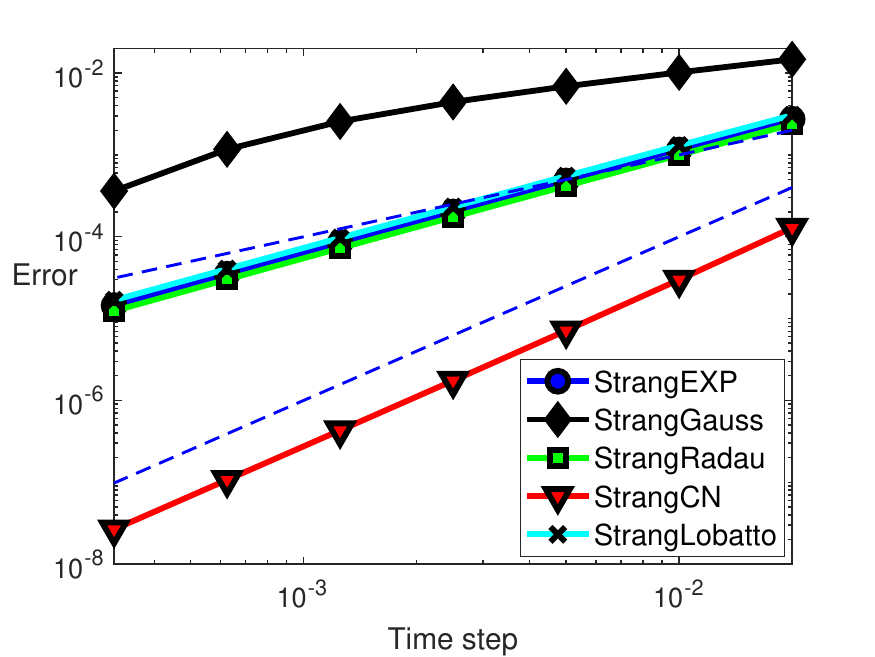} 
\caption{\textit{$u_{n+1}=\phi_\frac{\tau}{2}^f\circ\phi_\tau^D\circ\phi_\frac{\tau}{2}^f(u_n)$.}}\label{Fig:CompRKa}
\end{subfigure}%
\begin{subfigure}{.5\textwidth}
\centering
\includegraphics[width=\textwidth]{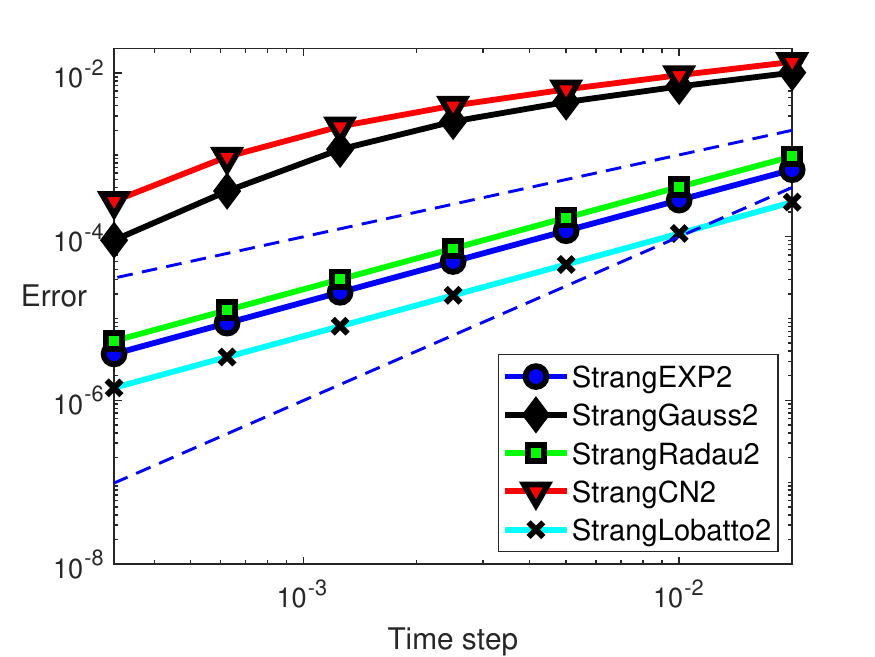}
\caption{\textit{$u_{n+1}=\phi_\frac{\tau}{2}^D\circ\phi_\tau^f\circ\phi_\frac{\tau}{2}^D(u_n)$.}}\label{Fig:CompRKb}
\end{subfigure}
\caption{\textit{The reduction of order of the Strang splitting~\eqref{eq:StrangfDf} method is not avoided when the following 2 stage Runge Kutta methods are used to solve the diffusion~\eqref{eq:diffusion}: Gauss (order 4), Radau 1a (order 3) or Lobatto 3c (order 2). The 1d problem considered is $\partial_t u=\partial_{xx} u+1$ with inhomogeneous Dirichlet boundary conditions. In addition, for the Strang splitting~\eqref{eq:StrangDfD}, solving the diffusion~\eqref{eq:diffusion} with the Crank-Nicolson method does not permit to remove the reduction of order. Reference slopes one and two are given in dashed lines.}}\label{Fig:CompRK}
\end{figure}

In Figure~\ref{Fig:CompRKa}, we approximate the diffusion part~\eqref{eq:diffusion} of the splitting~\eqref{eq:StrangfDf} with a variety of Runge-Kutta methods. We use the 2 stage Gauss method (order 4), the 2 stage Radau 1a method (order 3) and the 2 stage Lobatto 3c method (order 2) (see Appendix~\ref{Appendix:B} for the Butcher tableau of these methods). We also compute the error when the diffusion is solved exactly and when the Crank-Nicolson method is used, corresponding to the splitting method~\eqref{eq:StrangfDf} and~\eqref{eq:StrangfDfcn}. Except for Crank-Nicolson, none of the classical Runge-Kutta methods that we tested allows to remove the order reduction. We observe that the 2 stage Gauss method is the method for which the error is the largest, when in comparison the Crank-Nicolson method (equivalently the 1 stage Gauss method) is by far the method for which the error is the smallest for all considered time steps $\tau$.

In Figure~\ref{Fig:CompRKb} we apply the Strang splitting method~\eqref{eq:StrangDfD} instead of the Strang splitting~\eqref{eq:StrangfDf} to the problem~\eqref{eq:parabolic1}. The same experiment is then performed where we approximate the diffusion equation~\eqref{eq:diffusion} with different Runge-Kutta methods. We see that, for the Strang splitting~\eqref{eq:StrangDfD}, the Crank-Nicolson method does not allow to remove the order reduction. Surprisingly, it turns out that it is, amongst the methods tested, the scheme for which the error is the largest.

\subsection{Solution dependent source term}
\begin{figure}
\centering
\begin{subfigure}{.5\textwidth}
\centering
\includegraphics[width=\textwidth]{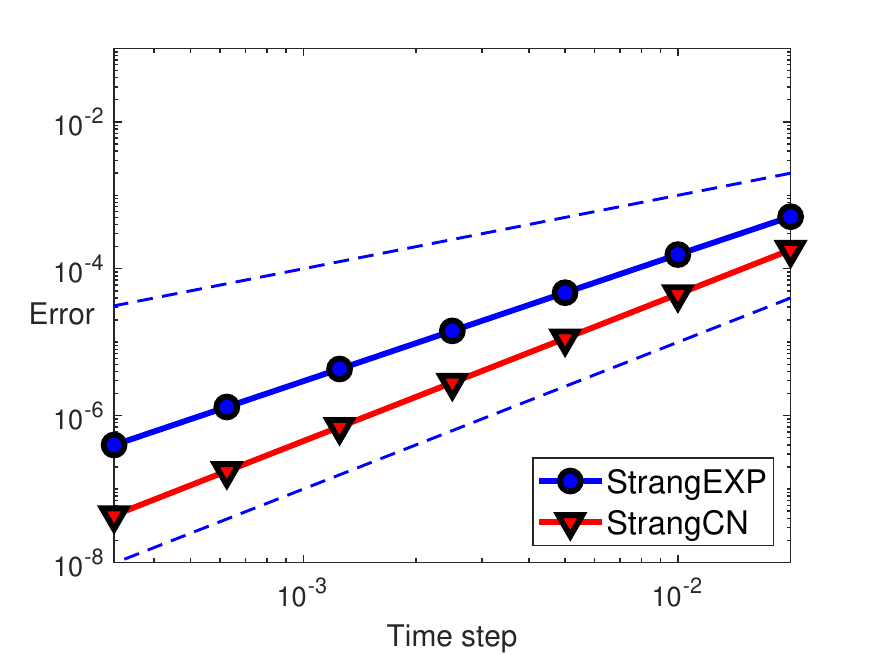} 
\caption{\textit{$f(u)=u$ with Robin boundary conditions.}}\label{Fig:fua}
\end{subfigure}%
\begin{subfigure}{.5\textwidth}
\centering
\includegraphics[width=\textwidth]{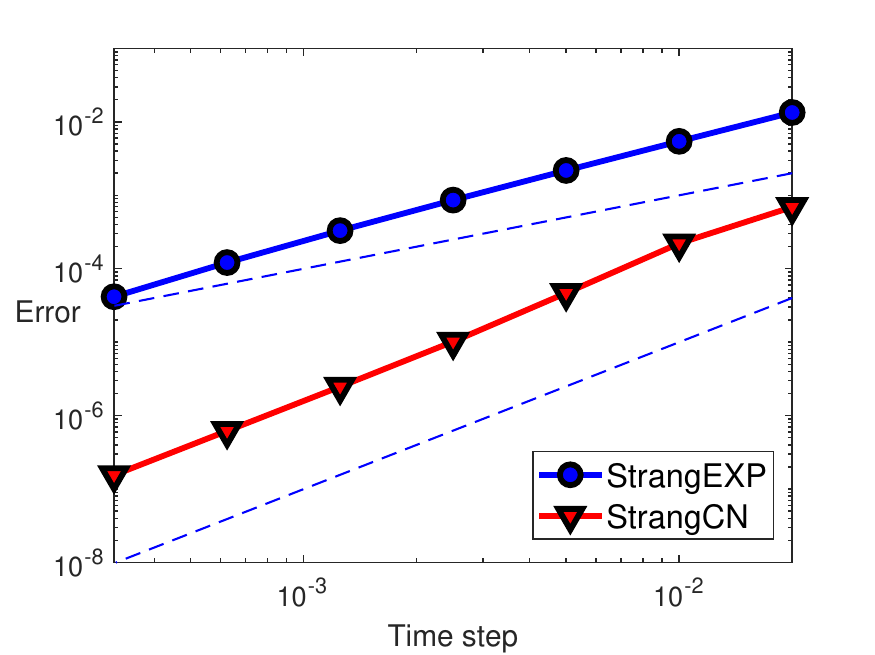}
\caption{\textit{$f(u)=u^2$ with mixed boundary conditions.}}\label{Fig:fub}
\end{subfigure}
\caption{\textit{The Strang splitting method~\eqref{eq:StrangfDf} has no order order reduction when the diffusion equation~\eqref{eq:diffusion} is solved using the Crank-Nicolson scheme. On the left picture, the equation is $\partial_t u = \Delta u + u$ with Robin boundary conditions. On the right picture, the equation is $\partial_t u = \Delta u+u^2$ with Neumann boundary conditions on the left and bottom boundaries and Dirichlet boundary conditions on the top and right boundaries. Reference slopes one and two are given in dashed lines.}}\label{Fig:fu}
\end{figure}

In Figure~\ref{Fig:fua}, we consider the following two dimensional problem with Robin boundary conditions, for $(x,y)\in\Omega=(0,1)^2$, $t\in [0,T]$,
\begin{align}
\partial_t u(x,y,t) &= \partial_{xx} u(x,y,t)+\partial_{yy}u(x,y,t)+u(x,y,t),\nonumber\\
u(0,y,t)+\partial_n u(0,y,t)&=y^2, \quad
u(1,y,t)+\partial_n u(1,y,t)=y^2+2,\nonumber \\
u(x,0,t)+\partial_n u(x,0,t) &= x^2, \quad
u(x,1,t)+\partial_n u(x,1,t)  = x^2+2,                          \nonumber\\
u(x,y,0)&=x^2+y^2.
\end{align}\label{eq:u}
As already observed, we see that the splitting \emph{StrangEXP} suffers from order reduction, when in comparison the splitting~\eqref{eq:StrangfDfcn} is of order two.

In Figure~\ref{Fig:fub}, we consider the following problem, for $(x,y)\in\Omega=(0,1)^2$, $t\in [0,T]$, 
\begin{align}
\partial_t u(x,y,t) &= \partial_{xx} u(x,y,t)+\partial_{yy}u(x,y,t)+u^2(x,y,t),\nonumber\\
\partial_n u(0,y,t)&=\frac{1}{2}, \quad
u(1,y,t)=\frac{\e^1+\e^y}{2},\nonumber \\
\partial_n u(x,0,t) &= \frac{1}{2}, \quad
u(x,1,t)  = \frac{\e^x+\e^1}{2},    \nonumber \\
u(x,y,0)&=\frac{\e^x+\e^y}{2}.\label{eq:u2}
\end{align}
For problem~\eqref{eq:u2}, the splitting~\eqref{eq:StrangfDfcn} is significantly more accurate than the splitting~\eqref{eq:StrangfDf}. In particular, using the Crank-Nicolson method to solve the diffusion part allows to increase the precision by a factor close to 1000 for the smallest time step $\tau=3.125\cdot 10^{-4}$.
\subsection{Stationary problems}

\begin{figure}
\centering
\begin{subfigure}{.5\textwidth}
\centering
\includegraphics[width=\textwidth]{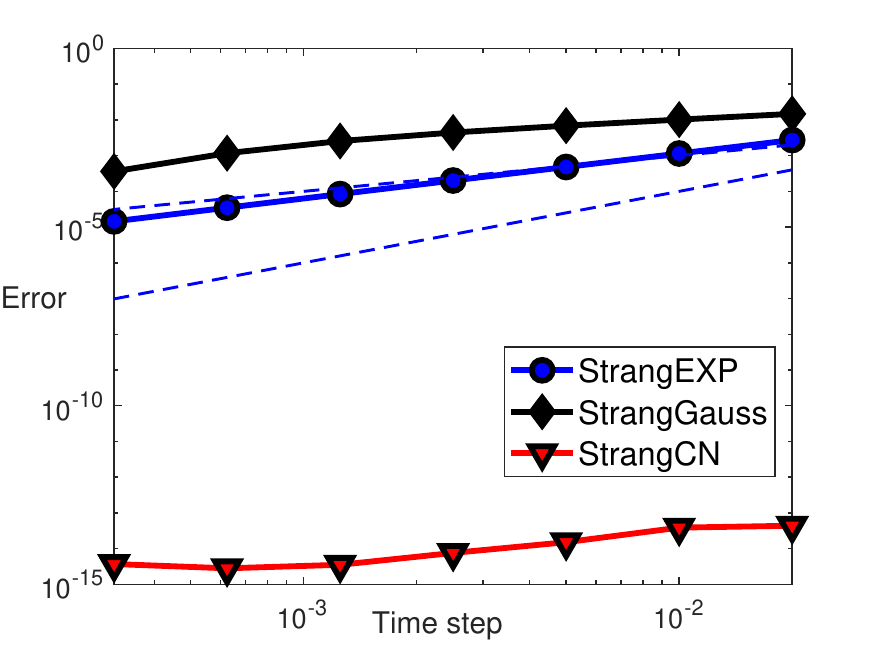} 
\caption{\textit{Error for a stationary problem}}\label{Fig:Stationary1}
\end{subfigure}%
\begin{subfigure}{.5\textwidth}
\centering
\includegraphics[width=\textwidth]{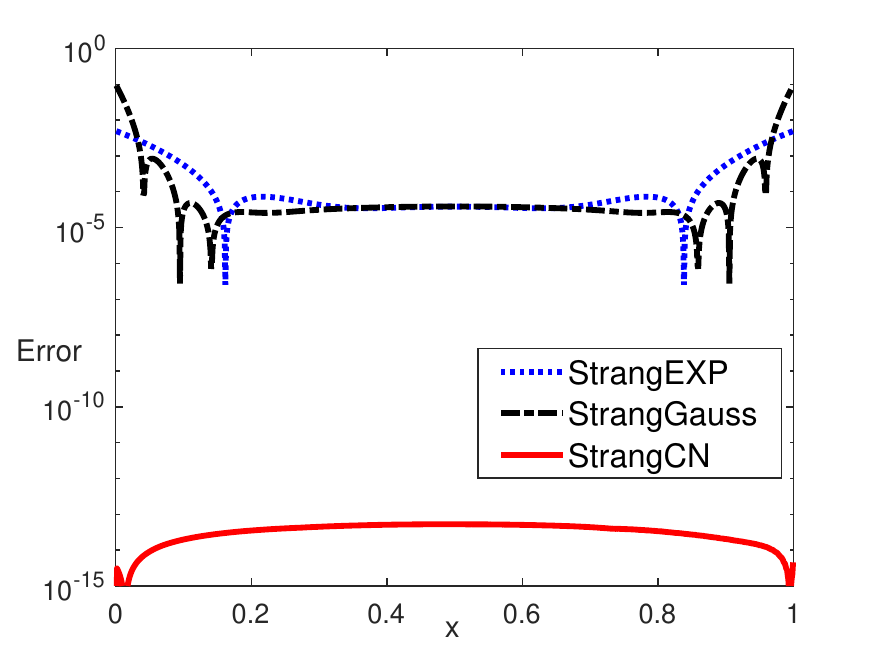}
\caption{\textit{$|u_n(x)-u(t_n,x)|$ for $\tau=10^{-2}$}}\label{Fig:Stationary2}
\end{subfigure}
\caption{\textit{The numerical solution given by the splitting~\eqref{eq:StrangfDfcn} has error close to machine precision for the stationary problem~\eqref{eq:stationary}. It contrast, for the splitting \emph{StrangEXP}~\eqref{eq:StrangfDf}, the error is closed to $10^{-5}$. When the two stage Gauss method is used to solve the diffusion equation~\eqref{eq:diffusion} instead of the Crank-Nicolson method, the error deteriorates. The problem considered is $ \partial_t u =\partial_{xx}u -1$ with $u_0=x^2/2$. Reference slopes one and two are drawn in dashed line in the left picture.}}\label{Fig:Stationary}
\end{figure}

In Figure~\ref{Fig:Stationary}, we consider the following stationary problem on $\Omega=(0,1)$, with $t\in [0,T]$, 
\begin{align}\label{eq:stationary}
\partial_t u(x,t) = \partial_{xx} u(x,t) - 1\quad \mathrm{in}\ (0,1), \qquad u(0,t)=0,\ u(1,t)=\frac{1}{2}, \qquad u(x,0)=\frac{x^2}{2}.
\end{align}
Since the problem is stationary ($u(x,t)=u(x,0)$), we take $u(x,t)=\frac{x^2}{2}$ as the reference solution. In Figure \ref{Fig:Stationary}, we observe that the error of the splitting~\eqref{eq:StrangfDfcn} is close to the machine precision ($10^{-15}$) and it does not decay for smaller time steps. In comparison, the error of the splitting~\eqref{eq:StrangDfD} is around $10^{-5}$ and it diminishes almost linearly when the time steps become smaller, i.e. the splitting~\eqref{eq:StrangDfD} suffers from order reduction down to one even for stationary problems. We observe the same phenomenon of order reduction for the two stage Gauss method with an error worse than the splitting~\eqref{eq:StrangfDf}. For the other Runge-Kutta methods considered previously, the error curves are nearly identical to the one of the splitting~\eqref{eq:StrangfDf} and they are not drawn for better readability.

\begin{figure}
\centering
\begin{subfigure}{.5\textwidth}
\centering
\includegraphics[width=\textwidth]{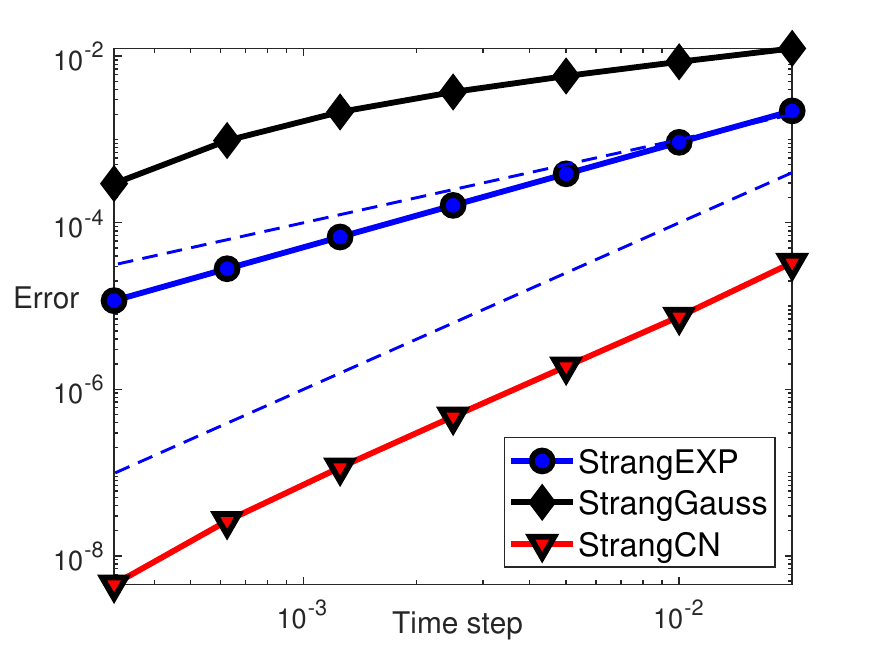} 
\caption{\textit{Error for a stationary problem with $f(u)=u$}}\label{Fig:Stationaryfu1}
\end{subfigure}%
\begin{subfigure}{.5\textwidth}
\centering
\includegraphics[width=\textwidth]{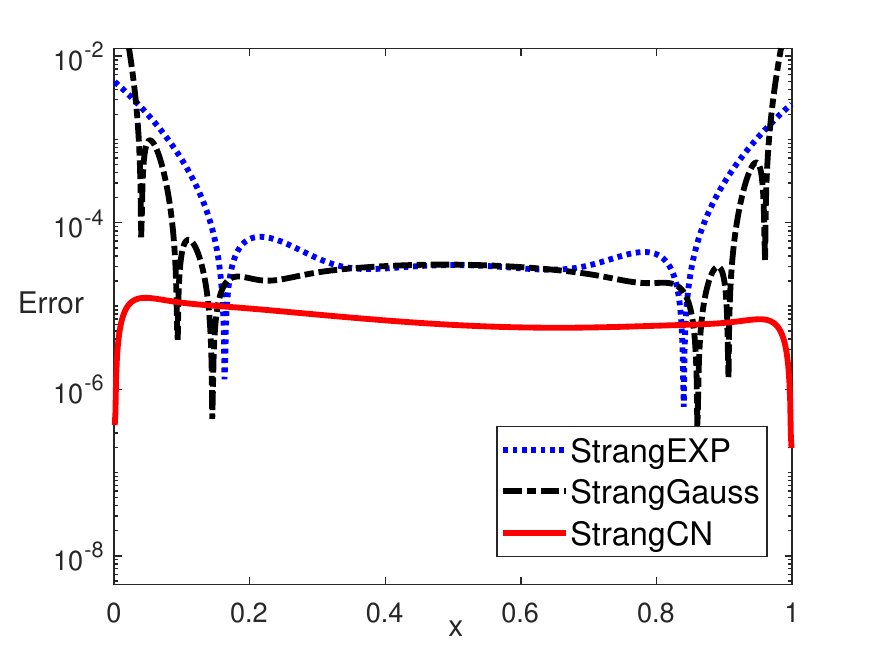}
\caption{\textit{$|u_n(x)-u(t_n,x)|$ for $\tau=10^{-2}$}}\label{Fig:Stationaryfu2}
\end{subfigure}
\caption{\textit{For a stationary problem with a source term that depends on the solution $u$, the splitting \emph{StrangCN} is not exact, in contrast with a problem where the source term f only depends on the space variable (see Figure~\ref{Fig:Stationary}). The parabolic problem considered is $\partial_t u = \partial_{xx}u+u,\ u_0=\cos(x)$, with inhomogeneous Dirichlet boundary conditions. Reference slopes one and two are given in dashed lines on the left picture. }}\label{Fig:Stationaryfu}
\end{figure}

In Figure~\ref{Fig:Stationaryfu}, we consider a stationary problem where the source term $f$ depends on the solution $u$, for $x\in [0,1]$ with $t\in [0,T]$, 
\begin{align}\label{eq:stationaryfu}
\partial_t u(x,t) = \partial_{xx} u(x,t) - u(x,t)\quad \mathrm{in}\ (0,1), \quad u(0,t)=1,\ u(1,t)=\cos(1), \quad u(x,0)=\cos(x).
\end{align}
The reference solution is $u(x,t)=\cos(x)$. We observe that the splitting \emph{StrangCN} is not exact in comparison with the previous stationary problem~\eqref{eq:stationary}, where $f$ only depends on the space variable. However, even for this problem where $f=f(u)$ that does not satisfy the hypotheses of Theorem~\ref{thm:fx}, we observe that the splitting $\emph{StrangCN}$ is second order convergent compared to the splittings $\emph{StrangEXP}$ and $\emph{StrangGauss}$, which have a reduced order of convergence between one and two. As seen in Figure~\ref{Fig:Stationaryfu2}, the loss of accuracy is mostly due to the large error made near the boundary of the domain. For better readability, we did not draw  the convergence curves for \emph{StrangLobatto} and \emph{StrangRadau} since they are superposed to the convergence curve of \emph{StrangEXP}.

We also observed numerically that the convergence analysis presented in this paper does not persist for dispersive problems, e.g. replacing $\partial_t u(x,t)$ by $i\partial_t u(x,t)$ formally in~\eqref{eq:parabolic1} to obtain a Schr\"odinger type problem (the convergence curves are not drawn for conciseness).
\paragraph{Acknowledgements.} This work was partially supported by the Swiss National Science
Foundation, grants No. 200020\_184614 and No. 200020\_178752. GB and GV would like to acknowledge
great hospitality when visiting Institut de Mathématiques de Toulouse
thanks to grant ANR project NABUCO, ANR-17-CE40-0025.

\appendix
\section{The Crank-Nicolson rational approximation of an analytic semigroup}\label{Appendix:A}
The aim of this Appendix is to prove Theorem~\ref{thm:hansbo}, which is a direct consequence of~\cite[Theorem 2.1]{Han99}. We present here a new direct and self contained proof of Theorem~\ref{thm:hansbo}. Similarly to~\cite{Han99}, the proof that we present holds for a general separable Banach space and not only for the special case $L^p(\Omega)$. This is useful in our context, if we consider a more abstract problem as described in Remark~\ref{rem:abstract}.

Let $X$ be a separable complex Banach space with norm $\|\cdot\|$. Let $A$ be a closed densely defined linear operator such that all eigenvalues are in a sector in the left half plane, or equivalently  for $\alpha\in (0,\frac{\pi}{2})$ the resolvent set $\rho(A)$ contains the set $\overline{\Sigma}_{\alpha}$:
\begin{equation}\label{eq:Sectorial1}
\rho(A)\supset \overline{\Sigma}_{\alpha}
\end{equation}
where $\Sigma_{\alpha}=\{z\in \mathbb{C}\ ;\ z\neq 0,\ |\mathrm{arg}(z)|<\pi-\alpha\}$.
Assume that for all $z\in \Sigma_{\alpha}$, the resolvent of $A$, $R(z, A)=(zI-A)^{-1}$, satisfies the following bound for the operator norm,
\begin{equation}\label{eq:Sectorial2}
\|R(z,A)\|\leq \frac{M}{|z|},
\end{equation}
where $M\geq 1$. Under those assumptions, the operator $A$ is the infinitesimal generator of an analytic semigroup given by integral formula (see~\cite[Definition I.3.4]{En00}),
\begin{align}\label{eq : ExpA}
\e^{t A}=\frac{1}{2\pi i}\int_\Gamma \e^{zt}R(z,A)\mathrm{dz},
\end{align}
with $t\geq 0$ and where $\Gamma $ is the boundary of $\Sigma_{\alpha}$ with imaginary part increasing along $\Gamma$. As before $r(z)$ denotes the stability function of the Crank-Nicolson scheme,
$$
r(z)=\frac{1+\frac{z}{2}}{1-\frac{z}{2}}.
$$
The purpose of this Appendix is to prove the following theorem :
\begin{theorem}\label{thm:CN}
Let $u_0\in \mathcal{D}(A)$ and let $n\geq 3$, then
$$
\|(r(\tau A)^n-\e^{\tau n A})u_0\|\leq \frac{C\tau^2}{t_n}\|A u_0\|,
$$
with $C$ is a positive constant independent of $u_0$, $\tau$, $n$, and $t_n=n\tau$. 
\end{theorem}
In the following lemma, we give some estimates for of $r(z)$. The proof is inspired from~\cite{CroPoly}.
\begin{lemma}
\begin{enumerate}
\item For all $z\in \mathbb{C}^-$ with $|z|\leq 1$, we have
\begin{equation}\label{lemma : 2}
\left|r(z)-\e^{z} \right|\leq \frac{5}{12}|z|^3.
\end{equation}
\item 
For all $z \in \C^-$, 
\begin{equation}\label{lemma : 3}
|r(z)|\leq
\max\big(e^{\frac45 \Re (z)},e^{\frac4{5\Re(z)}}\big)
\end{equation}
\item For all $y\in\mathbb{R^+}$ and for all integer $k\geq 1$, there exists $C_k$ independent of $y$ such that
\begin{equation}\label{lemma : 6}
\int_1^\infty \e^{-\frac{y}{x}}\frac{y^{k}}{x^{k+1}}\mathrm{dx}\leq C_k.
\end{equation}
\item For all $y\in\mathbb{R^+}$ and for all integer $k\geq 1$, there exists $C_k$ independent of $y$ such that
\begin{equation}\label{lemma : 7}
\int_0^1 \e^{-yx}y^{k+1}x^k\mathrm{dx}\leq C_k.
\end{equation}
\end{enumerate}
\end{lemma}
\begin{proof}
\begin{enumerate}
\item Since $z\neq-\frac{1}{2}$, we have
\begin{align*}
r(z)-\e^{z}&=\frac{z^3}{4}\frac{1}{1-\frac{z}{2}}-z^3\int_0^1 \e^{(1-s)z}\frac{s^2}{2}\mathrm{ds}.
\end{align*}
Therefore, for $z\in \mathbb{C}^-$ and $|z|\leq 1$, we obtain
$$
\left|r(z)-\e^{z}\right|\leq |z^3|\left(\frac{1}{4}+\int_0^1\frac{s^2}{2}\mathrm{ds}\right)=\frac{5}{12}|z^3|.
$$
\item We first show that 
\begin{equation}\label{eq:rminusx}
r(-x)\leq \e^{-x}
\end{equation} 
for $x\in \mathbb{R}^+$. Since $\e^{-x}$ and $r(-x)$ are both equal to $1$ for $x=0$, it suffices to show that $\frac{\mathrm{d}}{\mathrm{dx}}r(-x)\leq \frac{\mathrm{d}}{\mathrm{dx}}\e^{-x}$ for $x\geq 0$. A calulation yields $\frac{\mathrm{d}}{\mathrm{dx}}r(-x)=\frac{-1}{(1+\frac{x}{2})^2}\leq -\e^{-x}=\frac{\mathrm{d}}{\mathrm{dx}}\e^{-x}.$
For $\alpha\in (0,\frac{\pi}{2})$ and $0 \leq \rho\leq 1$, we have  
\begin{align*}
\left|r(-\rho\e^{\pm i\alpha})\right|^2=\frac{1+\frac{\rho^2}{4}-\rho\cos(\alpha)}{1+\frac{\rho^2}{4}+\rho\cos(\alpha)}\leq\left(\frac{1-\frac{2\rho}{5}\cos(\alpha)}{1+\frac{2\rho}{5}\cos(\alpha)}\right)^2=r(-\frac{4\rho}{5}\cos(\alpha))^2.
\end{align*}
Indeed, we {observe that}
\begin{align*}
&\frac{1+\frac{\rho^2}{4}-\rho\cos(\alpha)}{1+\frac{\rho^2}{4}+\rho\cos(\alpha)}\leq\left(\frac{1-\frac{4\rho}{5}\cos(\alpha)}{1+\frac{4\rho}{5}\cos(\alpha)}\right)^2\Leftrightarrow \frac{\frac{8}{5}+\frac{2\rho^2}{5}}{2+\frac{8\rho^2}{25}\cos(\alpha)^2}\cos(\alpha)\leq \cos(\alpha)
\end{align*}
and the last inequality is true since for $0\leq\rho\leq 1$ and $0\leq \cos(\alpha)\leq 1$, we deduce
$$
\frac{\frac{8}{5}+\frac{2\rho^2}{5}}{2+\frac{8\rho^2}{25}\cos(\alpha)^2}\leq \frac{\frac{8}{5}+\frac{2}{5}}{2}\leq 1.
$$
Using~\eqref{eq:rminusx}, we obtain
\begin{equation}\label{eq:rhoinf1}
\left|r(-\rho\e^{\pm i\alpha})\right|\leq \e^{-\frac{4\rho}{5}\cos(\alpha)}.
\end{equation}
If $z=-\rho\e^{\pm i\alpha}$, with $\alpha\in (0,\frac{\pi}{2})$ and $\rho\geq 1$, we start to observe that
$$
\left|r(z)\right|=\left|\frac{1+\frac{4}{2z}}{1-\frac{4}{2z}}\right|=\left|r\left(\frac{4}{z}\right)\right|.
$$
Since $\mathrm{arg}(\frac{4}{z})=-\mathrm{arg}(z)$, we write $\frac{4}{z}=\eta \e^{\mp i \alpha}$ with $\eta=\frac{4}{\rho}$. Since $\rho\geq 1$, we deduce $0<\eta\leq 4$.
We observe that
$$
\frac{1+\frac{\eta^2}{4}-\eta\cos(\alpha)}{1+\frac{\eta^2}{4}+\eta\cos(\alpha)}\leq\left(\frac{1-\frac{\eta}{10}\cos(\alpha)}{1+\frac{\eta}{10}\cos(\alpha)}\right)^2=r(-\frac{\eta}{5}\cos(\alpha))^2.
$$
Indeed, we have
\begin{align*}
&\frac{1+\frac{\eta^2}{4}-\eta\cos(\alpha)}{1+\frac{\eta^2}{4}+\eta\cos(\alpha)}\leq\left(\frac{1-\frac{\eta}{10}\cos(\alpha)}{1+\frac{\eta}{10}\cos(\alpha)}\right)^2
\Leftrightarrow \frac{\frac{2}{5}+\frac{\eta^2}{10}}{2+\frac{\eta^2}{50}\cos(\alpha)^2}\cos(\alpha)\leq \cos(\alpha)
\end{align*}
and the last inequality is true since for $0<\eta\leq 4$ and $0\leq \cos(\alpha)\leq 1$, we deduce
$$
 \frac{\frac{2}{5}+\frac{\eta^2}{10}}{2+\frac{\eta^2}{50}\cos(\alpha)^2}\leq \frac{\frac{2}{5}+\frac{16}{10}}{2}=\frac{20}{20}=1.
$$
Using~\eqref{eq:rminusx}, we obtain
$$
\left|r(-\eta\e^{\mp i\alpha})\right|\leq \e^{-\frac{\eta}{5}\cos(\alpha)}.
$$
This gives us for $\rho=\frac{4}{\eta}$,
$$
\left|r(-\rho\e^{\pm i\alpha})\right|\leq \e^{-\frac{4}{5\rho}\cos(\alpha)}.
$$
Together with~\eqref{eq:rhoinf1}, this concludes the proof.
\item The proof is made by induction.
For $k=1$, we obtain,
$$
\int_1^\infty \e^{-\frac{y}{x}}\frac{y}{x^2}\mathrm{dx}=\left.\e^{-\frac{y}{x}}\right|_1^\infty=1-\e^{-y}\leq 1.
$$
Assuming the result is true for $k$, we obtain
$$
\int_1^\infty \e^{-\frac{y}{x}}\frac{y^k}{x^{k+1}}\mathrm{dx}=\left.\e^{-\frac{y}{x}}\frac{y^{k-1}}{x^{k-1}}\right|_1^\infty+(k-1)\int_1^\infty\e^{-\frac{y}{x}}\frac{y^{k-1}}{x^{k}}\mathrm{dx}=-\e^{-y}y^{k-1}+C\leq C,
$$
where we used $\e^{-y}y^{k-1}\leq C$, with $C$ independent of $y$.
\item With the change of variable $\hat{x}=\frac{1}{x}$, we observe it is a direct consequence of the previous inequality~\eqref{lemma : 6}.
\end{enumerate}
\end{proof}
We introduce the following notation for the stability function of the implicit Euler method,
\begin{align*}
r_0(z)&=\frac{1}{1-z}.
\end{align*}
The rational approximation $r_0$ satisfy the following inequalities.
\begin{lemma}\label{lemma:IMP}
Let $z\in \mathbb{C^-}=\{z\in\mathbb{C}\ ;\ \Re(z)\leq 0 \}$ with $|z|\leq 1$, then 
$$
\left| r_0\left(\frac{z}{2}\right)-\e^{\frac{z}{2}}\right|\leq\frac{3}{8}|z|^2\quad \mbox{and}\quad
\left| r_0\left(\frac{z}{2}\right)^2-\e^{z}\right|\leq \frac{6}{8}|z|^2.
$$
\end{lemma}
\begin{proof}
Since $z\neq-\frac{1}{2}$, we have
\begin{align*}
r_0\left(\frac{z}{2}\right)-\e^{\frac{z}{2}}&=\frac{z^2}{4(1-\frac{z}{2})}-\frac{z^2}{4}\int_0^1 \e^{(1-s)z}s\mathrm{ds}.
\end{align*}
Therefore, for $z\in \mathbb{C}^-$ and $|z|\leq 1$, we obtain
$$
\left| r_0\left(\frac{z}{2}\right)-\e^{\frac{z}{2}}\right|\leq \frac{|z^2|}{4}\left(1+\int_0^1 s\mathrm{ds}\right)=\frac{3}{8}|z^2|.
$$
For the second inequality we observe that
$$
\left|r_0\left(\frac{z}{2}\right)^2-\e^{z}\right|=\left|r_0\left(\frac{z}{2}\right)\left(r_0\left(\frac{z}{2}\right)-\e^{\frac{z}{2}}\right)+\left(r_0\left(\frac{z}{2}\right)-\e^{\frac{z}{2}}\right)\e^{\frac{z}{2}}\right|\leq \frac{6}{8}|z^2|,
$$
which concludes the proof.
\end{proof}
We define $F_n$ and $G_n$ as follows,
\begin{align*}
F_{n}(z)=r(z)^{n-2}r_0\left(\frac{z}{2}\right)^2-\e^{\tau (n-1)z}\quad \mbox{and}\quad G_{n}(z)=r(z)^{n-1}r_0\left(\frac{z}{2}\right)-\e^{\tau (n-\frac{1}{2})z}.
\end{align*}
The term $r(z)^{n-2}r_0\left(\frac{z}{2}\right)^j$, $j=1,2$ correspond to the stability function of the composition of the Crank-Nicolson scheme with respectively one or two half steps of the Euler implicit scheme. This provides higher regularity for the solution since $r_0\left(\frac{\tau A}{2}\right)^j:\mathcal{D}(A^k)\rightarrow \mathcal{D}(A^{k+j})$ for $k$ any integer. For more general results on composition of rational approximation, see~\cite{Han99}.
\begin{lemma}
Let $n\geq 2$. Then $F_n(A)$ and $G_n(A)$ satisfy the following integral formula
$$
F_n(A)=\frac{1}{2\pi i}\int_\Gamma F_n(z)R(z,A)\mathrm{dz},\quad\mbox{and}\quad
G_n(A)=\frac{1}{2\pi i}\int_\Gamma G_n(z)R(z,A)\mathrm{dz},
$$
where $\Gamma=\{z\in\mathbb{C}\ ; \ |\arg(z)|=\pi-\alpha\}$.
\end{lemma}
\begin{proof}
The rational function $r(z)^{n-2}r_0\left(\frac{z}{2}\right)^2$ is holomorphic in a neighbourhood of the spectrum of $A$ and it vanishes at infinity . Therefore (see~\cite[Theorem VII.9.4]{Dun88}), we have 
$$
r(A)^{n-2}r_0\left(\frac{A}{2}\right)^2=\frac{1}{2\pi i}\int_\Gamma r(z)^{n-2}r_0\left(\frac{z}{2}\right)^2R(z,A)\mathrm{dz}.
$$ 
We conclude the proof using~\eqref{eq : ExpA}. The proof for $G_n(A)$ is similar and thus omitted.
\end{proof}
The proof that follows is inspired from~\cite[Theorem 9.3]{Tho06}.
\begin{proposition}\label{prop:compositionCN}
For $n\geq 3$ and $u_0\in X$, we have
\begin{align}\label{eq : rr02}
\left\|r(\tau A)^{n-2}r_0\left(\frac{\tau A}{2}\right)^2u_0-\e^{\tau (n-1) A}u_0\right\|\leq C\frac{\tau^2}{t_n^2}\|u_0\|
\end{align}
and
\begin{align}\label{eq : rr0}
\left\|r(\tau A)^{n-1}r_0\left(\frac{\tau A}{2}\right)u_0-\e^{\tau (n-\frac{1}{2}) A}u_0\right\|\leq C\frac{\tau}{t_n}\|u_0\|,
\end{align}
where $C$ is a constant independent of $u_0$, $\tau$, and $n$.
\end{proposition}
\begin{proof}
Since $\frac{1}{n}=\frac{\tau}{t_n}$, we need to show that
$$
\|F_n(A)\|\leq \frac{CM}{n^2}.
$$
Let $z=-\rho\e^{\pm i\alpha}$ with $\rho\geq 1$. We have from inequality~\eqref{lemma : 3} that 
$$
|r_1(-\rho\e^{ \pm i \alpha })|\leq \e^{-\frac{c}{\rho}},
$$
where $0<c< 1$ denotes a constant. Additionally, we have
$$
\left|r_0\left(-\frac{\rho\e^{ \pm i \alpha }}{2}\right)\right|=\frac{1}{\sqrt{1+\frac{\rho^2}{4}+\rho\cos(\alpha)}}\leq \frac{2}{\rho}.
$$
We recall that for all $x\in \mathbb{R^+}$, we have
\begin{equation}\label{lemma : 5}
\e^{-x}\leq \frac{C}{x^p},
\end{equation}
where $C$ is independent of $x$.
Therefore, since $|\e^{-(n-1)\rho\e^{  i \alpha }}|\leq\e^{-(n-1)\rho\cos(\alpha)}\leq \frac{C}{(n-1)^2\rho^2}$, we obtain for $n\geq 3$,
\begin{align*}
|F_n(-\rho\e^{\pm  i \alpha })|&\leq \e^{-\frac{c(n-2)}{\rho}}\frac{4}{\rho^2}+ \frac{C}{(n-1)^2\rho^2}\leq \left(\e^{-\frac{c(n-2)}{\rho}}\frac{(n-2)^2}{\rho^2}+ \frac{1}{\rho^2}\right) \frac{C}{n^2}.
\end{align*}
We use the inequality~\eqref{lemma : 6} with $y=c(n-2)$ and obtain
\begin{align*}
\int_1^\infty\e^{-\frac{c(n-2)}{\rho}}\frac{(n-2)^2}{\rho^3}\mathrm{d\rho}&=\frac{1}{c^2}\int_1^\infty\e^{-\frac{c(n-2)}{\rho}}\frac{c^2(n-2)^2}{\rho^3}\mathrm{d\rho}\leq C.
\end{align*}
This allows us to bound the integral,
$$
\int_1^\infty |F_n(-\rho\e^{\pm  i \alpha })|\|R(-\rho\e^{\pm  i \alpha },A)\|\mathrm{d\rho}\leq \frac{C}{n^2}\int_1^\infty\left(\e^{-\frac{c(n-2)}{\rho}}\frac{(n-2)^2}{\rho^2}+ \frac{1}{\rho^2}\right) \frac{M}{\rho}\mathrm{d\rho}\leq \frac{CM}{n^2},
$$ 
using $\int_1^\infty \frac{1}{\rho^3}\mathrm{d\rho}=\frac{1}{2}$ and~\eqref{eq:Sectorial2}.

For $\rho\leq 1$, we write
$$
F_n(z)=r_0\left(\frac{z}{2}\right)^2\left(r(z)^{n-2}-\e^{(n-2)z}\right)+\left(r_0\left(\frac{z}{2}\right)^2-\e^{z}\right)\e^{(n-2)z}.
$$
We observe that
$$
r(z)^{n-2}-\e^{(n-2)z}=\left(r(z)-\e^{z}\right)\sum_{k=0}^{n-3}r(z)^{n-k-3}\e^{kz}.
$$
Since, by estimate~\eqref{lemma : 2}, $|r(-\rho\e^{\pm i \alpha })-\e^{-\rho\e^{ \pm i \alpha }}|\leq C\rho^3$ and since, by the inequality~\eqref{lemma : 3}, $|r(-\rho\e^{ \pm i \alpha })|\leq \e^{-\rho c}$, we obtain
$$
\left|r\left(-\rho\e^{\pm i\alpha}\right)^{n-2}-\e^{-(n-2)\rho\e^{\pm i \alpha }}\right|\leq C\rho^3(n-2)\e^{-\rho (n-3) c}\leq C\rho^3(n-3)^3\e^{-\rho (n-3) c} \frac{C}{n^2}.
$$
Using inequality~\eqref{lemma : 7} with $y=(n-3)c$, we deduce
\begin{align*}
&\int_0^1\rho^2(n-3)^3\e^{-\rho (n-3) c}\mathrm{d\rho}=\frac{1}{c^3}\int_0^1\rho^2(n-3)^3c^3\e^{-\rho (n-3)c}\mathrm{d\rho}\leq C.
\end{align*}
From Lemma~\ref{lemma:IMP}, we also obtain,
$$
\left|\left(r_0\left(\frac{\rho}{2}\e^{\pm i\alpha}\right)^2-\e^{-\rho\e^{\pm  i \alpha }}\right)\e^{-(n-2)\rho\e^{ \pm i \alpha }}\right|\leq C\rho^2\e^{-(n-2)\rho\cos(\alpha)}\leq \frac{C}{n^2}\rho^2(n-2)^2\e^{-(n-2)\rho\cos(\alpha)}.
$$
Using inequality~\eqref{lemma : 7} with $y=(n-2)\cos(\alpha)$, we obtain,
\begin{align*}
&\int_0^1\rho(n-2)^2\e^{-(n-2)\rho\cos(\alpha)}\mathrm{d\rho}=\frac{1}{\cos(\alpha)^2}\int_0^1\rho(n-2)^2\cos(\alpha)^2\e^{-(n-2)\rho\cos(\alpha)}\mathrm{d\rho}\leq C.
\end{align*}
Therefore, since $\left|r_0(\frac{\rho}{2}\e^{\pm i \alpha})\right|\leq 1$,
\begin{align*}
&\int_0^1|F_n(-\rho\e^{\pm  i \alpha })|\|R(-\rho\e^{\pm  i \alpha },A)\|\mathrm{d\rho}\\
&\leq \int_0^1\rho^2(n-3)^3\e^{-\rho (n-3) c}+\rho(n-2)^2\e^{-(n-2)\rho\cos(\alpha)}\mathrm{d\rho}\frac{CM}{n^2} \leq \frac{CM}{n^2}
\end{align*}
This concludes the proof for the first inequality~\eqref{eq : rr02}.
The second inequality~\eqref{eq : rr0} is obtained similarly and thus omitted.
\end{proof}
With the help of Proposition~\ref{prop:compositionCN}, we can now prove Theorem~\ref{thm:CN}.
\begin{proof}[Proof of Theorem~\ref{thm:CN}]
We write
\begin{align*}
&(r(\tau A)^n-\e^{\tau n A})u_0\\
&=(r(\tau A)^{n-1}r_0\left(\frac{\tau}{2}A\right)-\e^{\tau (n-\frac{1}{2})A})(1+\frac{\tau}{2}A)u_0+(1+\frac{\tau}{2}A-\e^{\frac{\tau}{2}A})\e^{\tau(n-\frac{1}{2})}u_0\\
&=(r(\tau A)^{n-2}r_0\left(\frac{\tau}{2}A\right)^2-\e^{\tau (n-1)A})(1+\frac{\tau}{2}A)u_0+(1+\frac{\tau}{2}A-\e^{\frac{\tau}{2}A})\e^{\tau(n-1)}u_0\\
&+\frac{\tau}{2}(r(\tau A)^{n-1}r_0\left(\frac{\tau}{2}A\right)-\e^{\tau (n-\frac{1}{2})A})Au_0+(1+\frac{\tau}{2}A-\e^{\frac{\tau}{2}A})\e^{\tau(n-\frac{1}{2})}u_0.
\end{align*}
From Proposition~\ref{prop:compositionCN}, we have
$$
\|(r(\tau A)^{n-2}r_0\left(\frac{\tau}{2}A\right)^2-\e^{\tau (n-1)A})(1+\frac{\tau}{2}A)u_0\|\leq C\frac{\tau^2}{t_n^2}(\|u_0\|+\tau\|A u_0\|)
$$
and
$$
\|\frac{\tau}{2}(r(\tau A)^{n-1}r_0\left(\frac{\tau}{2}A\right)-\e^{\tau (n-\frac{1}{2})A})Au_0\|\leq C\frac{\tau^2}{t_n}\|Au_0\|.
$$
We observe that 
$$
(1+\frac{\tau}{2}A-\e^{\frac{\tau}{2}A})=-\frac{\tau^2}{4}A^2\varphi_2(\frac{\tau}{2} A),
$$
where $\varphi_2(\tau A)$ is the bounded operator given by
$$
\varphi_2(z)=\int_0^1 \e^{(1-s)z}s\mathrm{ds}.
$$
Using the smoothing property of $\e^{t A}$, we obtain
$$
\|(1+\frac{\tau}{2}A-\e^{\frac{\tau}{2}A})\e^{\tau(n-1)A}u_0\|\leq \frac{\tau^2}{4}\|\varphi_2(\frac{\tau}{2} A)\|\|A\e^{\tau(n-1)A}\|\|Au_0\|\leq  \tau^2\frac{C}{t_{n-1}}\|Au_0\|
$$
and
$$
\|(1+\frac{\tau}{2}A-\e^{\frac{\tau}{2}A})\e^{\tau(n-\frac{1}{2})}u_0\|\leq \frac{\tau^2}{4}\|\varphi_2(\frac{\tau}{2} A)\|\|A\e^{\tau(n-\frac{1}{2})}\|\|Au_0\|\leq  \tau^2\frac{C}{t_{n-\frac{1}{2}}}\|Au_0\|.
$$
This concludes the proof of Theorem~\ref{thm:CN}.
\end{proof}
\section{Runge Kutta methods Butcher tableau}\label{Appendix:B}
We give the Butcher tableau and the stability function of the A-stable implicit Runge-Kutta methods considered in the numerical experiments (see also \cite[Chapter IV.5]{Hai10}). 
\begin{align*}
&\text{The two stage Gauss method (order 4):} & &\\
&\begin{array}{c|cc}
\frac{1}{2}-\frac{\sqrt{3}}{6} &\frac{1}{4}&\frac{1}{4}-\frac{\sqrt{3}}{6}\\[2mm]
\frac{1}{2}+\frac{\sqrt{3}}{6} & \frac{1}{4}+\frac{\sqrt{3}}{6}&\frac{1}{4}\\[2mm]
\hline\\[-3mm]
& \frac{1}{2} & \frac{1}{2}
\end{array} \quad&
R(y)&=\frac{1+\frac{y}{2}+\frac{y^2}{12}}{1-\frac{y}{2}+\frac{y^2}{12}}\\
&\text{The two stage Radau 1a method (order 3):} & & \\
&\begin{array}{c|cc}
0 &\frac{1}{4}&-\frac{1}{4}\\ [2mm]
\frac{2}{3} & \frac{1}{4}&\frac{5}{12}\\[2mm]
\hline\\[-3mm]
& \frac{1}{4} & \frac{3}{4}
\end{array}&
R(y)&=\frac{1+\frac{y}{3}}{1-\frac{2y}{3}+\frac{y^2}{6}}\\
&\text{The two stage Lobatto 3c method (order 2):} & & \\
&\begin{array}{c|cc}
0 &\frac{1}{2}&-\frac{1}{2}\\[2mm]
1 & \frac{1}{2}&\frac{1}{2}\\[2mm]
\hline\\[-3mm]
& \frac{1}{2} & \frac{1}{2}
\end{array} &
R(y)&=\frac{1}{1-y+\frac{y^2}{2}}.
\end{align*}
\bibliographystyle{abbrv}
\bibliography{biblioUA}
\end{document}